\documentclass[11pt, reqno]{amsart}
\usepackage{amsmath}
\usepackage{amsthm}
\usepackage{amssymb}
\usepackage{amscd}
\usepackage{amsbsy}

\usepackage{epsfig}
\usepackage{graphicx}
\usepackage{psfrag}
\usepackage{enumitem}
\usepackage{bbm}
\usepackage{hyperref}

\newtheorem{theorem}{Theorem}[section]
\newtheorem{remark}[theorem]{Remark}
\newtheorem{proposition}[theorem]{Proposition}
\newtheorem{lemma}[theorem]{Lemma}
\newtheorem{corollary}[theorem]{Corollary}
\newtheorem{definition}[theorem]{Definition}

\def\H#1{{\bf #1}}

\newfont\bbf{msbm10 at 12pt}
\def\eps{\varepsilon}
\def\R{{\mathbb R}}
\def\C{{\mathbb C}}
\def\N{{\mathbb N}}

\def\F{{\mathcal F}}

\def\D{{\mathcal D}}
\def\H{{\mathcal H}}
\def\A{{\mathcal A}}

\def\T{{\mathcal T}}

\def\sm{\setminus}

\def\tag#1{\hfill \qquad  #1}

\def\bd{\partial }
\def\le{\leqslant}
\def\ge{\geqslant}

\newcommand{\essup}{{\rm ess\, sup}}
\newcommand{\essinf}{{\rm ess\, inf}}

\newcommand{\beq}{\begin{equation}}
\newcommand{\eeq}{\end{equation}}

\newcommand\restr[2]{{
  \left.\kern-\nulldelimiterspace 
  #1 
  \vphantom{\big|} 
  \right|_{#2} 
  }}

\def\rauz{\mathcal{R}}
\def\zoret{\mathcal{Z}}
\def\aalpha{a}
\def\bbeta{b}
\def\M{\mathcal{M}}

\begin{document}
\hspace{-4cm}
\title[Recurrence for IETs and the Teichm\"uller flow]{Recurrence statistics for the space of Interval Exchange maps and the Teichm\"uller flow on the space of translation surfaces}

\author{Romain Aimino}\address{
Dipartimento di Matematica, II Universit\`a di Roma (Tor Vergata), Via della Ricerca Scientifica, 00133 Roma, Italy
}
\email{\href{mailto:aimino@mat.uniroma2.it}{aimino@mat.uniroma2.it}}
\urladdr{\url{http://www.mat.uniroma2.it/~aimino/}}

\author{Matthew Nicol}\address{
Department of Mathematics,
University of Houston,
Houston Texas,
USA}
\email{\href{mailto:nicol@math.uh.edu}{nicol@math.uh.edu}}
\urladdr{\url{http://www.math.uh.edu/~nicol/}}

\author{Mike Todd}\address{
Mathematical Institute,
University of St Andrews,
North Haugh,
St Andrews,
Fife,
KY16 9SS,
Scotland}
\email{\href{mailto:m.todd@st-andrews.ac.uk}{m.todd@st-andrews.ac.uk}}
\urladdr{\url{http://www.mcs.st-and.ac.uk/~miket/}}

\thanks{RA was supported by Conseil R\'egional Provence-Alpes-C\^ote d'Azur, by the ANR-project {\em Perturbations}, by
the PICS (Projet International de Coop\'eration Scientifique), Propri\'et\'es statistiques des syst\`emes dynamiques d\'eterministes et al\'eatoires, with the University of Houston, n. PICS05968 and by the European Advanced Grant Macroscopic Laws and Dynamical Systems (MALADY) (ERC AdG 246953). Most of this work was done when RA was affiliated to Aix Marseille Universit\'e, CNRS, CPT, UMR 7332, 13288 Marseille, France and Universit\'e de Toulon, CNRS, CPT, UMR 7332, 83957 La Garde, France. MN was partially supported by NSF grant  DMS 1101315 and by the French CNRS with a {\em poste d'accueil} position  at the Center of Theoretical Physics in Luminy. MT was partially supported by NSF grant DMS 1109587. RA and MN would like to thank Huyi Hu for discussions on Quasi-H\"{o}lder space and aperiodicity. MN would like to thank Mark Pollicott for interesting and helpful discussions concerning Rauzy-Veech renormalization and Teichm\"uller flow. The authors wish to  thank Sandro Vaienti for  helpful remarks, encouragement and  many useful  discussions  concerning this work.}

\date{\today}
\maketitle

\begin{abstract}
In this note we show that  the transfer operator of a Rauzy-Veech-Zorich renormalization map acting on a  space of quasi-H\"older functions is quasicompact
and derive certain statistical recurrence properties for this map and its associated Teichm\"uller flow. We establish  Borel-Cantelli lemmas, Extreme Value statistics and return time statistics for the map and flow. Previous results have established quasicompactness
 in H\"older or analytic function spaces, for example the work of M.~Pollicott and T.~Morita.  The quasi-H\"older function space is particularly useful for 
 investigating return time statistics. In particular we establish the shrinking target property for nested balls in the setting of Teichm\"uller flow.
Our point of view, approach and terminology derive from the work of M.~Pollicott augmented by that of M.~Viana.  

\end{abstract}

\tableofcontents{}

\section{Background and notation}
\label{sec:background}

\subsection{Dynamical Borel-Cantelli Lemmas and other limit laws.}

Let $T \colon X \to X$ be  a measure-preserving transformation
of a probability space $(X, \mu)$. We assume $X$ is also a metric space equipped with a metric $d$.  Dynamical Borel-Cantelli lemmas concern the following
set of questions: suppose  $(A_n)$ is a sequence of sets
such that $\sum_{n} \mu (A_{n})=\infty$, does $T^{n} (x) \in A_{n}$ for infinitely many
values of $n$ for $\mu$ a.e.\ $x\in X$? One special example of this is the case 
where  $(A_{n})$ is a nested sequence of balls
about a point, a setting which is often called the shrinking target
problem.

We let $S_n =
\sum_{j=0}^{n-1} 1_{A_j} \circ T^j$ and $E_n = \int_X S_n \, d \mu = \sum_{j=0}^{n-1} \mu(A_j)$. The property $\lim_{n\to \infty}
\frac{S_n (x)}{E_n}=1$ for $\mu$ a.e.\ $x\in X$ is often called the
Strong Borel--Cantelli (SBC) property in contrast to the Borel--Cantelli (BC)
property that $S_n (x)$ is unbounded for $\mu$ a.e.\ $x\in X$.

In the setting of uniformly hyperbolic systems pioneering work has been
done by W. Philipp~\cite{Philipp}, Kleinbock and Margulis~\cite{Kleinbock_Margulis}, Chernov and
Kleinbock~\cite{Chernov_Kleinbock} and Dolgopyat~\cite{Dolgopyat} (for uniformly partially 
hyperbolic systems). 

More recently dynamical Borel-Cantelli results have been proved for certain non-uniformly hyperbolic systems
 by example by Kim~\cite{Kim}, Gou\"{e}zel~\cite{Gouezel}, Gupta et al~\cite{GNO} and  Haydn et al~\cite{HNPV}. These works have also 
 yielded some interesting counterexamples. 
In the context of flows, Maucourant~\cite{Maucourant} has proved the analogous
Borel Cantelli  property for nested balls in the setting of geodesic flows.  Athreya~\cite{Athreya} gives large deviation and quantitative recurrence results for the 
Teichm\"uller geodesic flow.

Related to Borel-Cantelli lemmas are logarithmic laws for the shrinking target problem.  These results concern the asymptotic scaling behavior given by the 
limit
\[
\lim_{r\to 0} \frac{\tau_r (x,y)}{\mu (B_r (y))},
\]
 where $\tau_r (x,y)=\min \{ n: d(T^n x,y)<r\}$ and $B_r (y)$ is a  ball of radius $r$ about $y\in X$.

Of  relevance  to our setting is work of Masur~\cite{Mas93}, who proved a logarithm type law for 
Teichm\"uller geodesic flow on the moduli space of quadratic differentials and work of Galatolo and Kim~\cite{Galatolo_Kim} who
obtain Borel-Cantelli like results for generic interval exchange transformations. Marchese~\cite{Marchese1, Marchese2} also obtained related results on the shrinking target problem for the Rauzy-Veech-Zorich algorithm, with applications to a generalization of the Khinchin theorem for interval exchange transformations. 
He also obtained logarithmic limit laws for  returns for Teichmuller flow on translation surfaces~\cite[Theorem 1,3]{Marchese2}. These quantitative results apply under a logarithmic scaling, unlike our results which apply to the unscaled flow.

Statistical properties of the Teichm\"{u}ller flow  and the Rauzy-Veech-Zorich map have been investigated thoroughly in recent years. Avila, Gou\"{e}zel and Yoccoz~\cite{Avila_Gouezel_Yoccoz} have shown that the decay of correlations for the flow is exponentially fast for H\"{o}lder observables. The corresponding problem for the Rauzy-Veech-Zorich map has been studied by Bufetov and Avila in~\cite{Bufetov} and~\cite{Avila_Bufetov}, where the decay was proven to be exponential as well. The main ingredient of the proof of the latter result was the construction of a Young Tower~\cite{Young} with an exponential tail of return times. Building upon this fact and  work of Melbourne and Nicol~\cite{Melbourne_Nicol05}, Pollicott~\cite{Pollicott} proved the almost sure invariance principle for H\"{o}lder observables, both for the flow and the map. The almost sure invariance principle is a strong reinforcement of the central limit theorem, which was previously established by Bufetov~\cite{Bufetov}, and has several consequences, such as the law of iterated logarithm and the arcsine law. The large deviations principle for H\"{o}lder observables follows also directly from the existence of an exponential Young tower and results of Melbourne and Nicol~\cite{Melbourne_Nicol08}.

We also establish  recurrence statistics such as Poisson limit laws and Extreme Value Laws (EVLs) for Teichm\"uller flow, but
 we leave the detailed description of these properties and results to Section~\ref{extreme}.

\subsection{Interval Exchange Transformations}
\label{ssec:IET}

In  this section we synthesize the basic model described by Viana in \cite{Via08} with  the framework developed by Pollicott~\cite{Pollicott} (see also~\cite{Morita}).  
Pollicott's short paper~\cite{Pollicott} is a  very clear account of the Rauzy-Veech-Zorich induction and renormalization  from the viewpoint of hyperbolic dynamics. 
We begin by defining our dynamical systems.  This starts with interval exchange transformations, in particular focussing on the formalism described by Viana.  We then move to the Rauzy-Veech induction and renormalisation; the Zorich induction and renormalisation; and finally the Morita-Pollicott renormalisation.   We will point out the minor differences with Pollicott's framework as we go along, but broadly speaking, the difference here is that our induced maps are \emph{first} returns.  We relate these dynamical systems to the Teichm\"uller flow on the space of translation surfaces later on.

Following \cite[Chaper 1]{Via08}, let $I\subset \R$ be an interval and $\{I_\aalpha:\aalpha\in \A\}$ a partition of $I$ into intervals indexed by a finite alphabet $\A$ with $d\ge 2$ symbols.  An \emph{interval exchange transformation} (IET) is a bijective map $f=f_{(\pi, \lambda)}:I\to I$ which is a translation of each subinterval $I_\aalpha$, preserves Lebesgue measure and is determined by the following combinatorial and metric data:

\begin{enumerate}[label=({\alph*}),  itemsep=0.0mm, topsep=0.0mm, leftmargin=7mm]
\item A pair $\pi=(\pi_0, \pi_1)$ of bijections $\pi_\eps:\A\to \{1, \ldots, d\}$ which describe the ordering of the subintervals $I_\aalpha$ before and after the action of $f$:
\begin{equation*}
\begin{pmatrix} \aalpha_1^0 &\aalpha_2^0 &\dots&\aalpha_d^0\\
\aalpha_1^1 &\aalpha_2^1 &\dots&\aalpha_d^1\\
\end{pmatrix}
\end{equation*}
where $\aalpha_j^\eps=\pi_\eps^{-1}(j)$ for $\eps\in \{0, 1\}$ and $j\in \{1, 2, \ldots, d\}$.

\item A vector $\lambda=(\lambda_\aalpha)_{\aalpha\in \A}$ of non-negative entries which represent the lengths of the subintervals $(I_\aalpha )_{\aalpha\in \A}$.
\end{enumerate}

We have a more detailed description of the intervals $I_\aalpha$ above which will be useful later: for $\eps\in \{0, 1\}$, let $I_\aalpha^{\pi_\eps}$ be the interval of length $\lambda_{\pi_\eps(\aalpha)}$ in position $\pi_\eps(\aalpha)$ in the interval $[0, \sum_{\aalpha}\lambda_a]$, where `position'  means starting at zero and counting to the right.

The transformation $p:=\pi_1\circ\pi_0^{-1}$ is called the \emph{monodromy invariant} of the pair $\pi=(\pi_0, \pi_1)$.  As Viana points out, we can always change our pair $\pi=(\pi_0, \pi_1)$ and rearrange the ordering of our lengths so that the resulting data $\pi'=(\pi_0', \pi_1')$ and  $\lambda'=(\lambda_\aalpha')_{\aalpha\in \A}$ represents the same IET as the one above, but with $\pi_0=id$.   Indeed, this is what is described in Pollicott's notes: moreover he always assumes that $\sum_\aalpha\lambda_\aalpha=1$.  However, the setup described here gives a slightly more complicated, but more flexible way for us to describe later dynamics.

The IET can now be described more explicitly as a translation.  For $\aalpha\in \A$, define 
$$w_\aalpha:=\sum_{\{\bbeta:\pi_1(\bbeta)<\pi_1(\aalpha)\}}\lambda_\bbeta-\sum_{\{\bbeta:\pi_0(\bbeta)<\pi_0(\aalpha)\}}\lambda_\bbeta.$$
Then
$$f_{(\pi, \lambda)}(x)=x+\sum_\aalpha w_\aalpha\cdot \mathbbm{1}_{I_\aalpha}(x).$$
Later it will be useful to think of the translation vector $w_a$ as $\sum_{\bbeta\in \A}\M_{ \aalpha\bbeta}\lambda_\bbeta$ where the $(a,b)$ entry of the matrix $\M$ is defined by
\begin{equation*}
\M_{\aalpha \bbeta} = \begin{cases}
 +1 & \text{ if } \pi_1(\bbeta)<\pi_1(\aalpha) \text{ and } \pi_0(\bbeta)>\pi_0(\aalpha),\\
-1 & \text{ if } \pi_1(\bbeta)<\pi_1(\aalpha) \text{ and } \pi_0(\bbeta)<\pi_0(\aalpha),\\
 \hspace{2mm} 0 & \text{ otherwise}.
 \end{cases}
 \end{equation*}

\subsection{Rauzy-Veech induction and renormalisation}

As is common for families of dynamical systems with parabolic-type behaviour, one way to proceed is to define a good renormalization scheme on the space of parameters.  In this setting this was pioneered by Masur and Veech.  Given a representative $(\pi, \lambda)$ of an IET, for $\eps\in \{0, 1\}$, let $\aalpha(\eps)$ denote the last symbol in the expression for $\pi_\eps$, i.e., $\aalpha(\eps)=\pi_\eps^{-1}(d)=\aalpha_d^\eps$.  Assuming the generic situation where $I_{\aalpha(0)}$ and $I_{\aalpha(1)}$ have different lengths, we say that
\begin{equation*}
(\pi, \lambda) \text{ has } \begin{cases}
\text{ type 0 if }& \lambda_{\aalpha(0)}>\lambda_{\aalpha(1)},\\
\text{ type 1 if }& \lambda_{\aalpha(0)}<\lambda_{\aalpha(1)}.
\end{cases}
\end{equation*}
Now set 
\begin{equation*}
J=  \begin{cases} I\sm f_{(\pi, \lambda)}(I_{\aalpha(1)})
&\text{if } (\pi, \lambda) \text{ has  type } 0,\\
I\sm I_{\aalpha(0)}
&\text{if } (\pi, \lambda) \text{ has  type } 1.
\end{cases}
\end{equation*}
Then the \emph{Rauzy-Veech induction} $\hat\T_0$ is defined as the first return by $f_{(\pi, \lambda)}$ to $J$.  Another way of viewing this, from which we see that we obtain a new IET of the form we started with (although with shorter total length of our intervals), is that $\hat \T_0(\pi, \lambda)=(\pi', \lambda')$ where, if $(\pi, \lambda)$ is type 0 then
\begin{equation*}
\begin{pmatrix} \pi_0'\\
\pi_1'\\
\end{pmatrix} =\begin{pmatrix} \aalpha_1^0&\dots \aalpha_{k-1}^0 &\aalpha_k^0 & \aalpha_{k+1}^0 &\dots&\dots& \aalpha(0)\\
\aalpha_1^1&\dots \aalpha_{k-1}^1 &\aalpha(0) & \aalpha(1) & \aalpha_{k+1}^1 &\dots&\aalpha_{d-1}^1\\
\end{pmatrix}
\end{equation*}
and  $\lambda'=(\lambda_\aalpha')_{\aalpha\in \A}$ for
$$\lambda_\aalpha'=\lambda_\aalpha \text{ for } \aalpha\neq\aalpha(0), \text{ and } \lambda_{\aalpha(0)}'= \lambda_{\aalpha(0)}- \lambda_{\aalpha(1)}.$$
Similarly, if 
 $(\pi, \lambda)$ is type 1 then
\begin{equation*}
\begin{pmatrix} \pi_0'\\
\pi_1'\\
\end{pmatrix} =\begin{pmatrix} \aalpha_1^0&\dots \aalpha_{k-1}^0 &\aalpha(1) & \aalpha(0) & \aalpha_{k+1}^0 &\dots& \aalpha_{d-1}^0\\
\aalpha_1^1&\dots \aalpha_{k-1}^1 &\aalpha_k^0 & \aalpha_{k+1}^0 &\dots &\dots&\aalpha(1)\\
\end{pmatrix}
\end{equation*}
and  $\lambda'=(\lambda_\aalpha')_{\aalpha\in \A}$ for
$$\lambda_\aalpha'=\lambda_\aalpha \text{ for } \aalpha\neq\aalpha(1), \text{ and } \lambda_{\aalpha(1)}'= \lambda_{\aalpha(1)}- \lambda_{\aalpha(0)}.$$

\begin{remark} \label{rmk:Theta}
This transformation on the set of lengths in $\R_+^{\A}$ can be expressed in terms of a matrix $\Theta$ given in (1.9) and (1.10) of \cite{Via08} and which consists only of 0s and 1s: in fact $\lambda'=\Theta^{-1 *}(\lambda)$ where $^*$ denotes the transpose.  $\Theta^{-1}$ is a non-negative matrix.
\end{remark}

We are interested in the set of $(\pi, \lambda)$ such that $\hat\T_0$ is defined for all time.  This occurs if and only if $(\pi, \lambda)$ satisfies the \emph{Keane condition}, which assumes  that
$$f_{(\pi, \lambda)}^n(\bd I_\aalpha)\neq \bd I_\bbeta \text{ for all } n\ge1 \text{ and } \aalpha, \bbeta\in \A \text{ with } \pi_0(\bbeta)\neq 1,$$ where $\bd I_\aalpha$ is the left endpoint of the subinterval $I_\aalpha$.
Moreover, if $(\pi, \lambda)$ satisfies the Keane condition then $f_{(\pi, \lambda)}$ is minimal (every $f_{(\pi, \lambda)}$-orbit is dense).
A pair $\pi=(\pi_0, \pi_1)$ is called \emph{reducible} if there exists $k\in \{1, \ldots, d-1\}$ such that $\pi_1\circ\pi_0^{-1}(\{1, \ldots, k\})=\{1, \ldots, k\}$.  In this case, $f_{(\pi, \lambda)}$ splits into two IETs with simpler combinatorics.  If $\pi$ is not reducible, we say it is \emph{irreducible}.  It can be shown that if $\lambda$ is rationally independent and $\pi$ is irreducible then $(\pi, \lambda)$ satisfies the Keane condition.  Keane conjectured that for fixed irreducible $\pi$, the map $f_{(\pi, \lambda)}$ was uniquely ergodic for almost-every $\lambda$.  
This conjecture was proved independently by Masur~\cite{Mas82} and Veech~\cite{Vee82}. The method
of proof of Veech was based on a renormalization scheme.

Given a fixed $d$, as above, we define the Rauzy class $\rauz = \rauz(\pi)$ of a pair $\pi$ as the set of all pairs $\pi'$ for which there exist $n\ge 0$, $\lambda$ and $\lambda'$ with $\hat\T_0^n(\pi, \lambda) = (\pi', \lambda')$. They form a partition of the set of all pairs $\pi$.  Thus we think of $\hat\T_0$ acting on sets $\rauz\times \R_+^{\A}$. For $d=2$ and $d=3$ there is a unique Rauzy class, but for $d\ge 4$ there is more than one.  Again we refer the reader to \cite[Chapter 1]{Via08} for a nice description of these.

The \emph{Rauzy-Veech renormalization map} $\T_0$ is simply the transformation $\hat\T_0$ renormalised so that the total length of the resulting interval is 1: thus the multiplying factor is 
$$\frac1{1-\lambda_{\aalpha(1-\eps)}} \text{ when } (\pi, \lambda) \text{ is type } \eps.$$ 
That is $\T_0(\pi, \lambda)=(\pi', \lambda'')$ where $\lambda''=\frac{\lambda'}{1-\lambda_{\aalpha(1-\eps)}}.$
Thus $\T_0$ acts on the  $(d-1)$ dimensional simplex 
\[
\Delta=\Delta_{\A}:=\{ \lambda=(\lambda_1,\ldots, \lambda_d): \lambda_i >0, \lambda_1+\ldots + \lambda_d=1\}.
\]

We define $|\lambda |=\sum_{j=1}^d \lambda_j$, then $\T_0$ has the form 
\[
\T_0 (\pi, \lambda)=\left(\pi', \frac{\Theta^{-1 *} \lambda}{|\Theta^{-1 *} \lambda |}\right)
\]
where $\Theta$ is the matrix defined in Remark \ref{rmk:Theta}.

Setting 
\begin{equation}
\Delta_{\pi, \eps}:=\left\{ \lambda\in \Delta_{\A}:\lambda_{\aalpha(\eps)}>\lambda_{\aalpha(1-\eps)}\right\} \text{ for } \eps\in \{0,1\},
\label{eq:Delta pi eps}
\end{equation}
$\T_0:\{\pi\}\times \Delta_{\pi, \eps}\mapsto \{\pi'\}\times \Delta$ is a bijection: a nice Markov property.  This also implies that $\Theta$ is constant on each $\{\pi\}\times \Delta_{\pi, \eps}$.

As in work of Veech \cite{Vee82} (see also Masur \cite{Mas82}), $\T_0$  has an absolutely continuous
and invariant ergodic measure (acim) $\mu_0$,  which is infinite. $\T_0$ is not uniformly hyperbolic.

\subsection{Zorich induction and renormalisation} 

Zorich  produced accelerated versions of the Rauzy-Veech maps discussed above in order to improve the expansion properties of the system and ultimately to find absolutely continuous invariant probability measures.  For this subsection we fix a Rauzy class $\rauz$.  Now take $\pi=(\pi_0, \pi_1)$ in this class and $\lambda\in \R_+^{\A}$ satisfying the Keane condition.  Then for each $k\ge 1$ write $(\pi^k, \lambda^k)=\hat\T_0^k(\pi, \lambda)$ and let $\eps^k$ denote the type of $(\pi^k, \lambda^k)$ and $\eps$ denote the type of $(\pi,\lambda)$.  Then $n_1=n_1(\pi, \lambda)$ is defined as the smallest $k$ such that $\eps^k\neq \eps$ and the \emph{Zorich  induction} is defined by
$$\hat \T_1(\pi,\lambda)=\hat\T_0^{n_1}(\pi, \lambda).$$
Similarly, the \emph{Zorich renormalisation}  $\T_1:\rauz\times \Delta \to \rauz\times \Delta$ is defined as $\T_1=\T_0^{n_1}$.
This map has a Markov partition into countably many domains.  Indeed, let
$$\Delta_{\pi, \eps, n}:=\{\lambda\in \Delta_{\pi, \eps}:\eps^1=\cdots =\eps^{n-1}=\eps\neq \eps^n\}.$$
Then for each $\pi\in \rauz$, $\T_1: \{\pi\}\times \Delta_{\pi, \eps, n}\mapsto \{\pi^n\}\times \Delta_{\pi^n, 1-\eps}$ is a bijection.  Moreover, 
$$\lambda^n=c_n\Theta^{-n *}(\lambda),$$
where $c_n>0$ and $\Theta^{-n *}$ depends only on $\pi, \eps, n$. Let also $\Delta_{\eps} = \cup_{\pi \in \rauz} \Delta_{\pi, \eps}$ and $\Delta_{1 - \eps} = \cup_{\pi \in \rauz} \Delta_{\pi, 1 -\eps}$.

\begin{theorem}[Zorich]
For a given Rauzy class $\rauz$, $\T_1$ has an absolutely continuous invariant probability measure $\mu_1$.  Moreover, for $\eps\in\{0,1\}$,  
$$\T_1^2: \Delta_{\eps}  \to \Delta_{\eps}$$
is mixing with respect to the restriction to $\Delta_{\eps}$ of the measure $2 \mu_1$. Similarly $$\T_1^2: \Delta_{1 - \eps}  \to \Delta_{1-\eps}$$
is mixing with respect to $2\mu_1$ restricted to $\Delta_{1 - \eps}$.
\end{theorem}

As already noted above, $\T_1( \Delta_{\eps})=  \Delta_{1 - \eps}$, so the absolutely continuous invariant probability measure (acip) $\mu_1$ is not mixing, but has two cyclic classes.  

\subsection{Morita-Pollicott renormalisation }
\label{ssec:MorPol}
A common approach (see \cite{Avila_Gouezel_Yoccoz, Morita, Pollicott}) is to consider a map $\T_2$ derived from $\T_1$ further by 
inducing by first return times on an element of a  dynamical partition with compact closure in the parameter space. $\T_2$ has the advantage that it is a  multidimensional piecewise expanding map.  The setup in Pollicott~\cite{Pollicott} is slightly different to that outlined here, but for most practical purposes, it is identical.

Recalling the definition of $ \Delta_{\pi, 0},  \Delta_{\pi, 1}$ from \eqref{eq:Delta pi eps}, let $$\mathcal{P}=\{ \{\pi\}\times \Delta_{\pi, 0}, \{\pi\}\times \Delta_{\pi, 1}: \pi \in \rauz \}$$ be the usual finite partition of $\rauz\times\Delta$ and define for $n\ge 1$ 
\[
\mathcal{P}_n:=\bigvee_{k=0}^{n-1} \T_1^{-k} \mathcal{P}.
\]
Pollicott's approach is to choose an $n_B>1$ and a partition element $B\in \mathcal{P}_{n_B}$ such that $B$ has compact closure $\bar{B}$ contained in the open simplex $\R \times \Delta$. In this case, $B$  is the image of an inverse branch of $\T_1^{n_B}$ which is a 
strict contraction for the Hilbert metric (see also \cite[Corollary 1.21]{Via08}). 
Define $n_2(\pi, \lambda)$ to be the first return time of $(\pi, \lambda)\in B$ to $B$ under $\T_1$, i.e.
\[
n_2(\pi, \lambda)= \inf \{ k>0: \T_1^k (\pi, \lambda) \in B \}.
\]
Then
$\T_2: B \to B$ is defined as the induced first return time map under $\T_1$,
\[
\T_2 (\pi, \lambda)= \T_1^{n_2 (\lambda,\pi)} (\pi, \lambda).
\]

\begin{remark}\label{rmk:B placement}
Note that for each element $(\pi, \lambda)\in \rauz\times \Delta$, with $\lambda$ satisfying the Keane condition, we can find such a $B$ containing $(\pi, \lambda)$.
\end{remark}

The set $B$ has a natural countable partition $\mathcal{Q}=\{B_i\}_{i \in \mathcal{I}}$ into sets on which $n_2(\pi, \lambda)$ is constant.  The map $\T_2: B_i\to B$ is a diffeomorphism for
each $i \in \mathcal{I}$~\cite[Lemma 3.1]{Morita}.  $B$ has a naturally defined $\T_2$-invariant measure, namely $\mu_2:=\frac{\mu_1|_{B}}{\mu_1(B)}$. The density $h_B$ of 
$\mu_2$ with respect to Lebesgue measure on $B$  is strictly positive \cite[Lemma 2.3]{Pollicott} and analytic~\cite[Corollary 5.1.1]{Pollicott}. Let $\mathcal{Q}_n:=\bigvee_{k=0}^{n-1} \T_2^{-k} \mathcal{Q}$.

We have the following expansion and distortion properties. 

\begin{proposition}\cite[Lemma 2.2]{Pollicott} \label{prop:pollicott}
There exist $C>1$ , $\theta>1$ and $D_1$, $D_2$ such that for any $n\ge 1$ and any $x,y$ in the same element of $Q \in \mathcal{Q}_n$:

(1) $d(\T_2^n x, \T_2^n y ) \ge C\theta^n d(x,y)$;

(2) $ \left| \log \left( \frac{Jac(\T_2^n) (x)}{Jac(\T_2^n) (y)}\right)\right| \le D_1 d(\T_2^n x, \T_2^n y)$;

(3) $\frac{1}{D_2} \le \mu_2 (A)|J ac(\T^n_2)(x)| \le D_2$ for all $x\in A\in \mathcal{Q}_n$;

(4) $\T_2^n : Q \to B$ is a diffeomorphism.

\end{proposition}

\begin{remark}
Since there exists $c >0$ such that $c^{-1} \le h_B \le c$, we can also state the above point (3) using Lebesgue measure $m$ instead of $\mu_2$. (or more accurately, the product of the counting measure on $\rauz$ and Lebesgue measure on $\Delta$, even though we will always refer to this measure as Lebesgue)
\end{remark}

\subsection{Gibbs-Markov maps and their transfer operators} \label{subsection:gibbs}

The previous subsection motivates  a  more in depth study of  the following class of maps.

Let $(Y,d)$ be a compact metric space endowed with a probability measure $m$ with full support. Let $T : Y \to Y$ be a nonsingular measurable map.

We will say that $T$ is a Gibbs-Markov map if there exists a countable measurable partition $\mathcal{Q} = \{Y_i\}_{i \in \mathcal{I}}$ of $Y$ such that, if we denote by $\mathcal{Q}_n = \bigvee_{k=0}^{n-1} T^{-k} \mathcal{Q}$ the dynamical partition of $T^n$ and by $Jac(T^n)$ the jacobian of $T^n$ with respect to $m$ ( i.e. $m(T^nA) = \int_A Jac(T^n) \, dm$ for every subset $A \subset Y$ on which $T^n$ is injective), we have 

\begin{enumerate}

\item \label{assumption:bijection} $T^n  : Q \to Y$ is a bimeasurable bijection;

\item \label{assumption:expansion} $d(T^n x, T^n y) \ge C \theta^n d(x,y)$;

\item \label{assumption:distortion} $\left\vert \log \frac{Jac(T^n) (x)}{Jac(T^n)(y)} \right\vert \le D d(T^n x, T^n y)$;

\end{enumerate}

for all $n \ge 1$, all $Q \in \mathcal{Q}_n$ and all $x,y \in Q$, where $C,D >0$ and $\theta>1$ depend only on the map $T$.

It is well known such maps admit a spectral gap for their transfer operators on the space of H\"older functions. We will study spectral properties on a larger space which contains discontinuous functions, namely the Quasi-H\"older space, introduced by Keller \cite{Keller} and Saussol \cite{Saussol}. We recall the relevant definitions and properties, and refer to the aforementioned references for more details.

Let $\eps_0>0$, $0 < \alpha <1$ and $f : Y \to \mathbb{R}$ lie in $L^1_m(Y)$. We define the oscillation of $f$ on a Borel subset $S \subset Y$ by 
\[
\mbox{osc} (f,S)= {\essup}_S f -{\essinf}_S f.
\]
We define 
\[
|f|_{\alpha}:=\sup_{0<\eps\le \eps_0} \eps^{-\alpha}\int_Y \mbox{osc} (f, B_{\eps} (x) ) dm(x)
\]
and let $V_{\alpha}(Y) :=\{ f \in L^1_m (Y,\R): |f|_{\alpha} <\infty\}$.  This space is strictly larger than the space of H\"older functions of exponent $\alpha$ on $Y$ and in particular
contains characteristic functions of some measurable sets. If we define the norm $\|\cdot\|_{\alpha}:=|\cdot|_{\alpha}+\|\cdot\|_{L^1_m}$ then $V_{\alpha}(Y)$ is a Banach space.
Since $Y$ is compact, the space $V_{\alpha}(Y)$ is compactly embedded in $L^1_m (Y)$. Furthermore, $V_{\alpha}(Y)$ embeds continuously into $L^{\infty}_m(Y)$ and is a Banach algebra satisfying $|fg|_{\alpha} \le |f|_{\alpha} \|g\|_{\infty} + \|f \|_{\infty} |g|_{\alpha}$ for all $f,g \in V_{\alpha}(Y)$.

Note also that while $\| \cdot \|_{\alpha}$ depends on the choice of $\eps_0$, the space $V_{\alpha}(Y)$ does not, and two different $\eps_0$ give rise to two equivalent norms on $V_{\alpha}$.

Let $P$ denote the transfer operator of $T$ with respect to $m$. This is the $L^1$ adjoint of $T$ with respect to $L^\infty$, i.e. $\int_Y P \phi \, \psi \, dm = \int_Y \phi \, \psi \circ T \, dm$ for all $\phi \in L^1_m(Y)$ and $\psi \in L^{\infty}_m(Y)$.

The operator $P$ has the form $$P\phi(x) = \sum_{i \in \mathcal{I}} \frac{\phi(x_i)}{Jac(T)(x_i)},$$ where $x_i \in Y_i$ satisfies $T x_i = x$.

We will prove that the transfer operator $P$ of a Gibbs-Markov map $T$ is quasi-compact and admits a spectral gap on $V_\alpha(Y)$, from which it will follow exponential decay of correlations for $T$, for observables in $V_\alpha(Y)$. Our main tool will be a Lasota-Yorke type inequality (Lemma \ref{lem:LY}) and Hennion's theorem \cite{Hennion}. We refer to Baladi \cite{Baladi} for a systematic exposition of this approach.

The next technical lemma will also prove useful later. In order to state it, we need some more notations. For $Q \in \mathcal{Q}_n$, denote $I_{n, Q} : Y \to Q$ the inverse branch of the restriction of $T^n$ to $Q$. The transfer operator $P^n$ of $T^n$ has the form $$P^n \phi (x) = \sum_{Q \in \mathcal{Q}_n} g_n(I_{n,Q} x) \phi(I_{n,Q} x),$$ where $g_n = \frac{1}{Jac(T^n)}$.

Denote by $M_{n,Q}$ the operator defined on $L^1_m(Y)$ by $$M_{n,Q} \phi(x) = g_n(I_{n,Q} x) \phi(I_{n,Q} x).$$

\begin{lemma} \label{lem:mnq}
There exists $C > 0$ such that for any $n \ge 1$, $Q \in \mathcal{Q}_n$ and $\phi \in  V_{\alpha}(Y)$, we have $\| M_{n,Q} \phi \|_{L^1_m} = \int_Q | \phi | dm$ and $$\int_Y \mbox{osc}(M_{n,Q} \phi, B_{\eps}(x)) dm(x) \le C \int_Q \mbox{osc}(\phi, B_{c_{n,Q} \eps}(x)) dm(x) + C \eps \int_Q |\phi| dm, $$ where $c_{n,Q }$ is the Lipschitz constant of $I_{n,Q} : Y \to Q$.
\end{lemma}

\begin{proof}
The relation $\int_Y | M_{n,Q} \phi | dm = \int_Q  |\phi| dm$ follows from a change of variables. 

Observe that $\mbox{osc}(M_{n,Q} \phi, B_{\eps}(x)) = \mbox{osc}(g_n \phi,  I_{n,Q} B_{\eps}(x))$. Using \cite[Proposition 3.2 (iii)]{Saussol}, we have for all $x  \in Y$, $$ \mbox{osc}(M_{n,Q} \phi, B_{\eps}(x))  \le \mbox{osc}(\phi, I_{n,Q} B_{\eps}(x)) \underset{I_{n,Q} B_{\eps}(x)}\essup g_n + \mbox{osc}(g_n, I_{n,Q} B_{\eps}(x)) \underset{I_{n,Q} B_{\eps}(x)} \essinf |\phi|.$$

By the distortion control of assumption \ref{assumption:distortion}, we have $\underset{I_{n,Q} B_{\eps}(x)}\essup g_n \le C g_n(I_{n,Q} x)$ and, since $|e^t - e^s| \le |t -s| e^t$ for any $t, s \in \mathbb{R}$,
$$
\begin{aligned}
\mbox{osc}(g_n, I_{n,Q}B_{\eps}(x)) & \le \underset{{y,z \in I_{n,Q} B_\eps(x)}}\essup | \exp \log g_n(y) - \exp \log g_n(z) | \\ & \le \underset{{y,z \in I_{n,Q} B_\eps(x)}}\essup \left| \log \frac{g_n(y)}{g_n(z)} \right| g_n(y) \\ & \le D g_n(I_{n,Q}x) \underset{{y,z \in I_{n,Q} B_\eps(x)}}\essup  d(T^n y, T^nz) \\ & \le C g_n(I_{n,Q} x) \eps
\end{aligned}
$$
for some constant $C >0$. We also have $\mbox{osc}(\phi, I_{n,Q} B_{\eps}(x)) \le \mbox{osc}(\phi, B_{c_{n,Q} \eps}(I_{n,Q} x))$ and $\underset{I_{n,Q} B_{\eps}(x)}\essinf | \phi| \le | \phi( I_{n,Q} x) |$ for almost every $x \in Y$. Putting together all the above estimates, we get for almost every $x$, $$\mbox{osc}(M_{n,Q} \phi, B_{\eps}(x)) \le C \mbox{osc}(\phi, B_{c_{n,Q} \eps}(I_{n,Q}x)) g_n(I_{n,Q}x) + C \eps |\phi(I_{n,Q} x)| g_n(I_{n,Q} x).$$ After integration over $Y$, a change of variables finishes the proof.
\end{proof}

With this lemma, we can prove a Lasota-Yorke type inequality for $T$: 

\begin{lemma}\label{lem:LY}
If $\eps_0$ is sufficiently small then there exist $0<\eta<1$ and $C, D>0$ such that if $\phi \in V_{\alpha}(Y)$ then for all $n \ge 0$
\[
\|P^n \phi\|_{\alpha} \le C \eta^n \|\phi\|_{\alpha} +D \int_Y |\phi| dm.
\]
\end{lemma}

\begin{proof}
Since $P^n$ is a contraction on $L^1_m(Y)$ (see for instance Baladi \cite{Baladi}), it is sufficient to estimate $|P^n \phi|_{\alpha}$.  We will next apply Lemma~\ref{lem:mnq} to this operator, first noting that by assumption \ref{assumption:expansion}, $c_{n,Q} \le C \theta^{-n} \le C$, where $\theta > 1$. 
Writing $P^n = \sum_{Q \in \mathcal{Q}_n} M_{n,Q}$ and summing all the relations from Lemma \ref{lem:mnq}, \cite[Proposition 3.2 (i)]{Saussol} then implies that 
\begin{eqnarray*}
\int_Y \mbox{osc}(P^n \phi, B_{\eps}(x)) dm(x) \le C \int_Y \mbox{osc}(\phi, B_{C \theta^{-n} \eps}(x))dm(x) + C \eps \| \phi \|_{L^1_m}  \\
\le C \eps^{\alpha} \left( \theta^{-\alpha n} |\phi|_{\alpha} + \eps_0^{1-\alpha} \| \phi \|_{L^1_m} \right),
\end{eqnarray*}
for all $0 < \eps \le \frac{\eps_0}{C} = \eps_1$, so that $C \theta^{-n} \eps \le \eps_0$ and the bound $$ \int_Y \mbox{osc}(\phi, B_{C \theta^{-n} \eps}(x))dm(x) \le C \eps^{\alpha} \theta^{- \alpha n} | \phi |_{\alpha}$$ holds.

This shows $\|P^n \phi\|_{\alpha, \eps_1} \le C \theta^{- \alpha n} \|\phi\|_{\alpha, \eps_0} + C \|\phi\|_{L^1_m}$, where we put the subscript $\eps_0$ or $\eps_1$ in the notation for the Quasi-H\"older norm to emphasize the fact it was defined using either $\eps_0$ or $\eps_1$, and concludes the proof since the two norms $\| . \|_{\alpha, \eps_0}$ and  $\| . \|_{\alpha, \eps_1}$ are equivalent.
 \end{proof}

Classical arguments then allow us to prove exponential decay of correlations in the Quasi-H\"older norm:

\begin{proposition} \label{prop:decay_gibbs} There exists an unique absolutely continuous probability measure $\mu$ which is $T$-invariant, and its density $h$ belongs to $V_{\alpha}(Y)$. Furthermore, we have\begin{itemize}
\item[(a)] $\left\|P^n \phi - \left( \int_Y \phi \, dm \right) h \right\|_{\alpha} \le C \theta^n \|\phi \|_{\alpha}$;
\item[(b)] $\left|\int_Y \phi \, \psi \circ T^n \, d\mu- \int_Y \phi \, d \mu \int_Y \psi \, d \mu \right| \le C \theta^n \|\phi\|_{\alpha} \|\psi\|_{L^1_{\mu}}$, 
\end{itemize}
for all $n \ge 1$, for all $\phi \in V_{\alpha}$ and $\psi \in L^1(\mu)$, for some constants $C >0$ and $\theta <1$ which depend only on the map $T$.

\end{proposition}

\begin{proof}

Lemma~\ref{lem:LY} implies by Hennion's theorem \cite{Hennion} that $P$ is quasi-compact and has an essential spectral radius strictly less than 1 when acting on the space $V_{\alpha}(Y)$. To prove (a), it is then sufficient to prove that $1$ is a simple eigenvalue of $P$, and that there is no other eigenvalue on the unit circle. Let then $\phi \in V_{\alpha}$ be an eigenvector of $P$ for the eigenvalue $\lambda \in \C$ with $| \lambda | = 1$. From standard results, see for instance Aaronson \cite{Aaronson}, we know that $P$ has an essential spectral radius strictly less than $1$ when acting on the space of Lipschitz functions. This shows that $\phi$ is itself Lipschitz continuous, and then $\phi$ is a multiple of $h$ and $\lambda = 1$.

We now prove point (b): $$\begin{aligned} \int_Y \phi \, \psi \circ T^n d\mu- \int_Y \phi \, d \mu \int_Y \psi \, d\mu &= \int_Y \phi h \, \psi \circ T^n \,  dm - \int_Y \phi \, d \mu \int_Y \psi \, d \mu \\ &= \int_Y \left( P^n(\phi h) - \int_Y \phi h \,dm) \, h \right) \psi \, dm. \end{aligned}$$ Then, $ \left|\int_Y \phi \, \psi \circ T^n \, d\mu - \int_Y \phi \, d \mu \int_Y \psi \, d \mu \right| \le \left\|P^n(\phi h) - (\int_Y \phi h \, dm) h \right\|_{L^{\infty}_m} \| \psi \|_{L^1_m}.$ By (a), we have that $ \left\|P^n(\phi h) - (\int_Y \phi h \, dm) h \right\|_{L^{\infty}_m} \le C \theta^n \| \phi \|_{\alpha}$ since $V_{\alpha}(Y)$ embeds into $L^{\infty}_m$ and is a Banach algebra. On the other hand, $\| \psi \|_{L^1_m} \le c^{-1} \| \psi \|_{L^1_{\mu}}$ where $c = \inf h$ is strictly positive by Lemma 4.4.1 in \cite{Aaronson}. This proves (b).
\end{proof}

\section{Borel-Cantelli Lemmas}

\subsection{Borel-Cantelli lemmas for Gibbs-Markov maps}

We first investigate Borel-Cantelli lemmas for the map $\T_2$. From Proposition \ref{prop:pollicott}, we know $\T_2$ is a Gibbs-Markov map, so we will present general results for this class of maps.

Our result for Gibbs-Markov maps is a a fairly straightforward consequence of earlier work (see for example~\cite[Theorem 2.1]{Kim},~\cite[Proposition 2.6]{GNO}) and the description of their transfer operators we give in the previous subsection.

\begin{proposition}\label{prop:SBC_T2} Let $T$ be a Gibbs-Markov map on the compact metric space $(Y,d)$, as in the previous subsection, with absolutely continuous invariant measure $\mu$.
Let $\{\phi_n\}$ be a sequence of positive functions on $Y$ such that there exists  a constant $K>0$ with $\|\phi_n\|_{\alpha} \le K$ for  all $n$. Let $E_n=\sum_{j=1}^n \mu (\phi_j)$ and suppose $E_n$ is unbounded.
Then
\[
\lim_{n\to \infty}\frac{1}{E_n}  \sum_{j=1}^n \phi_j \circ T^j (x) \to 1
\]
for $\mu$~a.e.\ $x\in Y$.
\end{proposition}

The proof of this proposition, given below, is an easy consequence of a Gal-Koksma type law. 
We formulate this law as  a proposition of W. Schmidt~\cite{W1,W2} as stated
by Sprindzuk~\cite{Sprindzuk}:

\begin{proposition}\label{prop:sprindzuk}
  Let $(\Omega,\mathcal{B},\mu)$ be a probability space and let $f_k
  (\omega) $, $(k=1,2,\ldots )$ be a sequence of non-negative $\mu$
  measurable functions and $g_k$, $h_k$ be sequences of real numbers
  such that $0\le g_k \le h_k \le 1$, $(k=1,2, \ldots,)$.  Suppose
  there exists $C>0$ such that
  \begin{equation} \label{eq:sprindzuk}
    \tag{$*$} \int \left(\sum_{m<k\le n}( f_k (\omega) - g_k)
    \right)^2\,d\mu \le C \sum_{m<k \le n} h_k
  \end{equation}
  for arbitrary integers $m <n$. Then for any $\eps>0$
  \[
  \sum_{1\le k \le n} f_k (\omega) =\sum_{1\le k\le n} g_k   +
  O (\theta^{1/2} (n) \log^{3/2+\eps} \theta (n)
  )
  \]
  for $\mu$ a.e.\ $\omega \in \Omega$, where $\theta (n)=\sum_{1\le k
    \le n} h_k$.
\end{proposition}

\begin{proof}[Proof of Proposition~\ref{prop:SBC_T2}]
In Proposition~\ref{prop:sprindzuk} take  $f_k=\phi_k \circ T^k$, $g_k=h_k=\mu(\phi_k)$ and, using part (b) of Proposition \ref{prop:decay_gibbs}, calculate 
\begin{align*}
\Bigg|\sum_{i=m}^{n} \sum_{j=i+1}^n  \int \phi_j\circ  T^j \phi_i \circ & T^i d \mu -\mu(\phi_j) \mu (\phi_i)\Bigg| \\ 
&=\left|\sum_{i=m}^{n} \sum_{j=i+1}^n \int \phi_j\circ T^{j-i} \phi_i  -\mu(\phi_j) \mu (\phi_i)\right|\\
&\le \sum_{i=m}^n  \sum_{j=i+1}^n C_1\theta^{j-i} \| \phi_j\|_{\alpha} \|\phi_i\|_{L^1_{\mu}}\\
&\le C_2 \sum_{i=m}^n \|\phi_i\|_{L^1_{\mu}}.
\end{align*}
The result follows immediately from Proposition~\ref{prop:sprindzuk}.
\end{proof}

\begin{remark} \label{rmk:balls} For any measurable set $A \subset Y$, we have $\| \mathbbm{1}_A \|_\alpha \le m(A) + \sup_{0 < \eps \le \eps_0} \frac{m(B_\eps (\partial A))}{ \eps^\alpha}$. Hence, any sequence of sets $(A_n)$ such that for some $0 < \alpha \le 1$, $$\sup_n \sup_{0 < \eps \le \eps_0} \frac{m(B_\eps (\partial A_n))}{ \eps^\alpha} < \infty$$ and $\sum_n \mu(A_n) = \infty$ will satisfy the strong Borel-Cantelli property. In particular, the sequence does not need to be decreasing.

\end{remark}

\begin{remark} As a direct consequence, we get for the Morita-Pollicott renormalization map $\T_2 : B \to B$ the strong Borel-Cantelli for any sequence of positive functions $(f_n)$ on $B$ bounded in the space $V_\alpha(B)$ for some $0 < \alpha \le 1$, with $\sum_n \int f_n \, d \mu_2 = \infty$. Indeed, by Proposition \ref{prop:pollicott}, this map is Gibbs-Markov with respect to the partition $\mathcal{Q} = \{B_i\}_{i \in \mathcal{I}}$.
This applies in particular to any sequences of balls $(B_{r_n}(p_n))$ with $\sum_n \mu_2 (B_{r_n}(p_n)) = \infty$, since such sequences satisfy the condition of Remark \ref{rmk:balls} for $\alpha = 1$.
\end{remark}

\subsection{Borel-Cantelli lemmas for a class of non-uniformly expanding maps} We now turn to investigate Borel-Cantelli lemmas for the Rauzy-Veech-Zorich renormalization map $\T_1$. 

\begin{remark} 
Note that by Haydn et al~\cite[Theorem 6.1]{HNPV} or by Galatolo \cite[Lemma 6, Lemma 7]{Galatolo} if $\{U_n\}$ is a sequence of balls in $\Delta_{\pi,
\eps}$, $\eps\in \{0,1\}$, satisfying $\mu_1 (U_n) \ge \frac{C}{n}$ then 
$\T_1^{2n} (p) \in U_n$ i.o. for $\mu_1$ a.e.\ $p\in \Delta_{\pi,
\eps}$ since $(\T_1^2,\rauz\times\Delta, \mu_1)$ has exponential decay of correlations for
Lipschitz functions~\cite{Avila_Bufetov}.  We are interested in obtaining quantitative rates for this almost sure result.
\end{remark}

We first proceed to identify a class of maps containing $\T_1$ for which such results hold. 

Let $(X,d)$ be a bounded, locally compact and separable metric space, with a Borel finite positive measure $m$. Let $T : X \to X$ be a non-singular transformation for which $m$ is ergodic. 

Suppose there exists a compact subset $Y \subset X$ with $m(Y)>0$ (without loss of generality, we can assume $m(Y) = 1$) and a countable measurable partition $\mathcal{Q} = \{Y_i\}_{i \in \mathcal{I}}$ of $Y$ such that the first return time $$r(y) = \inf \{n \ge 1 \, : \, T^n y \in Y\}$$ of $T$ to $Y$ is constant on each $Y_i$, and the first return map $\widehat{T} = T^r : Y \to Y$ is Gibbs-Markov with respect to the partition $\mathcal{Q}$. We also assume the first return time is integrable with respect to $m$: $\int_Y r \, dm < \infty$.

We will refer to such systems as non-uniformly expanding maps, even though more general definitions exist in the literature.

Under these assumptions, there exists an unique absolutely continuous with respect to $m$ probability measure $\mu$ which is $T$-invariant, and the system $(X,T, \mu)$ is ergodic. The existence follows directly from the existence of such a measure for the first return map $\widehat{T}$ and the integrability of $r$, while the uniqueness is ensured by \cite[Theorem 1.5.6]{Aaronson}.

We will deduce a strong Borel-Cantelli property for decreasing sequences of functions supported in $Y$ from our result for Gibbs-Markov maps and the following result of Kim \cite[Theorem 3.1]{Kim}:

\begin{theorem}\label{BC_Kim}

Let $(X,T, \mu)$ be an ergodic measure-preserving transformation, and let $T_E : E \to E$ be the first return map to a set $E$ of positive $\mu$-measure. Let $(f_n)$ be a decreasing sequence of nonnegative functions supported in $E$ such that $\sum_n \int f_n d\mu = \infty$. If every subsequence $(f_{n_k})$ with $\sum_k \int f_{n_k} d \mu = \infty$ is strong Borel-Cantelli with respect to $T_E$, then $(f_n)$ is strong Borel-Cantelli with respect to $T$.
 \end{theorem}

As an immediate corollary of Proposition \ref{prop:SBC_T2} and Theorem \ref{BC_Kim}, we get:

\begin{theorem} \label{thm:sbc_nue}

Let $(X,T, \mu)$ be a non-uniformly expanding system as described above, with induced set $Y$. Then any sequence $(f_n)$ of positive functions, supported in $Y$, bounded in $V_\alpha(Y)$ for some $0 < \alpha \le 1$, with $\sum_n \int_Y f_n \, d \mu = \infty$, satisfies the strong Borel-Cantelli property.
\end{theorem}

As seen in subsection \ref{ssec:MorPol}, the Rauzy-Veech-Zorich renormalization map is a non-uniformly expanding map, with induced set $B$. Since by Remark \ref{rmk:B placement}, for any $p^* = (\pi, \lambda)$ satisfying the Keane condition, we can find a good induced set $B$ that contains $p$, we obtain: 

\begin{theorem}\label{SBC_T1}
Let  $U_n\subset \rauz\times\Delta$ be a decreasing sequence of balls, shrinking to a point $p^*$ which satisfies the Keane condition, such that $E_n:=\sum_{j=1}^n \mu_1 (U_j)$ diverges. Then, for $\mu_1$ almost every $p \in \rauz \times \Delta$ 
\[
\frac{1}{E_n} \sum_{j=1}^n \mathbbm{1}_{U_j} \circ \T_1^j (p) \to 1.
\]
\end{theorem}

\begin{proof} Set $f_n = \mathbbm{1}_{U_n}$. By the discussion above, for $n$ large enough, $f_n$ will be supported in some fixed good induced set $B$. Since, as in Remark~\ref{rmk:balls}, $(f_n)$ is bounded in $V_{\alpha}(B)$, it follows from Theorem \ref{thm:sbc_nue} that $(f_n)$ is strong Borel-Cantelli with respect to $\T_1$. \end{proof}

\begin{remark}
This result remains true for any decreasing sequence of sets $U_n$ shrinking to a point $p^*$ as soon as the boundaries of these sets are sufficiently regular to ensure the condition of Remark \ref{rmk:balls} is satisfied.
\end{remark}

We now consider more general, non necessarily decreasing, sequences of functions supported in the induced set $Y$. We will require additional properties for the non-uniformly expanding system, and we will see later they are satisfied by the Rauzy-Veech-Zorich map.

We set $C_n =  \{r=n\} \subset Y$. This set is a disjoint union of elements of $\mathcal{Q}$: we have $C_n = \cup_{i \in \mathcal{I}_n} Y_i$, where $\mathcal{I}_n = \{i \in \mathcal{I} \, : \, \restr{r}{Y_i} \equiv n\}$. 

\begin{definition} \label{def:good_nue} Let $T$ be a non-uniformly expanding map. We say $T$ is \emph{good} if 
\begin{enumerate}
\item $(X,T, \mu)$ is mixing;
\item $m(r > n) \le C \gamma^n$;
\item $c_i \le C \gamma^n$, for all $n \ge 1$ and $i \in \mathcal{I}_n$;
\end{enumerate}
for some $C >0$ and $\gamma < 1$, where  $c_i = c_{1, Y_i}$ is the Lipschitz constant of $I_i = I_{1, Y_i} : Y \to Y_i$, the inverse branch of $\widehat{T}$ restricted to $Y_i$.
\end{definition}

Note that  $(X,T, \mu)$ is mixing if and only if $ \mbox{gcd}\{\restr{r}{Y_i} \, : \, i \in \mathcal{I} \} = 1$, see e.g. \cite{Young99}.

Under these assumptions, we have the following result for the decay of correlations of $(X,T,\mu)$ for observables supported in $Y$:

\begin{theorem} \label{thm:decay_nue}

If $T$ is a good non-uniformly expanding map, there exist $0 < \kappa < 1$ and $C>0$ such that for all $\phi \in V_{\alpha}(Y)$ and all $\psi \in L^1(\mu)$ supported in $Y$,

\[
\left| \int_X \phi \, \psi \circ T^n \, d\mu - \int_X \phi \, d \mu \int_X \psi \, d \mu \right| \le C \kappa^n \| \phi\|_{\alpha} \| \psi \|_{L^1_\mu}.
\]

\end{theorem}

This theorem has the following corollary:

\begin{corollary} \label{cor:bc_nue} Let $T$ be a good non-uniformly expanding map. Suppose $\{\phi_n\}$ is  a sequence of positive functions with support  in $Y$ bounded in $V_\alpha(Y)$
with $E_n:=\sum_{j=1}^{n}  \mu (\phi_j)$  divergent. Then
\[
\frac{1}{E_n} \sum_{j=1}^n \phi_j \circ T^j (x) \to 1
\]
for $\mu$ a.e.\ $x\in X$.

\end{corollary}

\begin{proof}
We will use  Proposition~\ref{prop:sprindzuk}. Take $f_k=\phi_k\circ T^k$ and $h_k=g_k= \mu (\phi_k)$. A rearrangement of terms shows that it 
suffices to show 
\[
\sum_{i=m}^n \sum_{j=i+1}^{n} \mu (\phi_j \circ T^{j-i} \phi_i)  - \mu_1 (\phi_j) \mu (\phi_i) \le C \sum_{i=m}^n \mu (\phi_i).
\]
But $ | \mu (\phi_j \circ T^{j-i}\phi_i ) -\mu (\phi_j)\mu (\phi_i)| \le C \kappa^{j-i}  \|\phi_i  \|_{L^1_\mu}$ which yields
the result as $\sum_{j>i} \kappa^{j-i } $ is summable.
\end{proof}

To prove Theorem \ref{thm:decay_nue}, we will use operator renewal theory, in the spirit of Sarig \cite{Sarig} and Gou\"ezel \cite{Gouezel3}, even though in our situation of exponential tails for the return time, the situation is easier. We will make use of the following Proposition:

\begin{proposition} \cite[Proposition 3.4]{Gouezel2}\label{exp_renewal}

Let $Q$ be a Banach space and suppose $(R_n)_{n \ge 1}$ is a sequence of bounded operators on $Q$. Assume that $\|R_n\| = O(\theta^n)$ for some $0 < \theta < 1$. Hence $R(z)=\sum R_n z^n$ and $R'(z) =\sum n R_n z^{n-1}$ are well-defined operators on $Q$ for $z$ in the unit complex disc $\bar{\D}$.  Assume $1$ is a simple isolated eigenvalue of $R(1)$ and the  eigenprojector $\Pi$ satisfies $\Pi R'(1) \Pi =\gamma \Pi$ for some  $\gamma \not =1$ and that $I-R(z)$ is invertible for all $z\in \bar{\D} \setminus \{1\}$.
Let $V_n=\sum_{l=1}^{\infty} \sum_{k_1 +\ldots + k_l = n} R_{k_l}\circ \ldots \circ R_{k_1}$. Then $V_n$ is a bounded linear operator on $Q$
and $\|V_n -  \frac{1}{\gamma} \Pi \| = O(\kappa^n)$ for some $0 < \kappa < 1$.

\end{proposition}

Let $L$ be the transfer operator associated to the non-uniformly expanding map $T : X \to X$, defined for $\phi \in L^1(m)$ by $$L\phi(x) = \sum_{Ty = x} \frac{\phi(y)}{Jac(T)(y)}.$$

Let $P$ be the transfer operator associated to the first return map $\widehat{T} : Y \to Y$. By the results of subsection \ref{subsection:gibbs}, this operator admits a spectral gap on the space $V_\alpha(Y)$.

Let $R_n \phi : =\mathbbm{1}_{Y} L^n (\mathbbm{1}_{C_n} \phi)$ and $V_n \phi :=\mathbbm{1}_{Y} P^n (\mathbbm{1}_{Y} \phi)$. The linear operator $R_n$ corresponds to first returns to $Y$ at 
time $n$ while  $V_n$ considers all  points starting in $Y$ which have returned to $Y$ at time $n$, whether first return or not. The following renewal
equation holds:
\[
V_n=\sum_{l=1}^{\infty} \sum_{k_1 +\ldots + k_l = n} R_{k_l}\circ \ldots \circ R_{k_1}.
\]

We will show these operators satisfy the three required conditions to apply Proposition \ref{exp_renewal}. Recall the definition of a good non-uniformly expanding map from Definition \ref{def:good_nue}.

\begin{lemma} \label{lemma:ren1}

There exists $0<\theta < 1$ and $C >0$ such that $\|R_n \| \le C \theta^n$.

\end{lemma}

\begin{proof}

We have $R_n \phi = \sum_{i \in \mathcal{I}_n} \frac{\phi(I_i x)}{{\rm Jac}(\widehat{T})(I_i x)}$, whence $R_n = \sum_{i \in \mathcal{I}_n} M_{1, Y_i}$. Thus, by lemma \ref{lem:mnq}, we have $$\begin{aligned} \|R_n \phi \|_{L^1_m} \le \sum_{i \in \mathcal{I}_n} \|M_{1, Y_i} \phi \|_{L^1_m} = \sum_{i \in \mathcal{I}_n} \int_{Y_i} | \phi | \, dm& = \int_{C_n} | \phi | \, dm  \\ & \le m(C_n) \| \phi \|_{L^{\infty}_m} \\ & \le C m(C_n) \| \phi \|_{\alpha}, \end{aligned}$$ and $$\begin{aligned} \int \mbox{osc} (R_n \phi, B_{\eps} (x)) dm(x)  & \le \sum_{i \in \mathcal{I}_n} \int \mbox{osc}(M_{1, Y_i} \phi, B_{\eps} (x)) dm(x)  \\ & \le C \left( \sum_{i \in \mathcal{I}_n} \int_{Y_i}  \mbox{osc}( \phi, B_{c_i \eps} (x)) dm(x) + \eps \sum_{i \in \mathcal{I}_n} \int_{Y_i} | \phi | \, dm \right) \\ & \le C \int_{C_n} \mbox{osc} (\phi, B_{c^{(n)} \eps}(x)) dm(x) + C \eps \int_{C_n} | \phi | \, dm, \end{aligned} $$ where $c^{(n)} = \sup_{i \in \mathcal{I}_n} c_i$.

We have $$\begin{aligned} \int_{C_n} \mbox{osc}(\phi, B_{c^{(n)} \eps} (x))dm(x) \le \int_B \mbox{osc}(\phi, B_{c^{(n)} \eps} (x))dm(x) & \le (c^{(n)})^{\alpha} \eps^{\alpha} | \phi |_{\alpha} \\ & \le (c^{(n)})^{\alpha} \eps^{\alpha} \| \phi \|_{\alpha} \end{aligned}$$ and $\int_{C_n} | \phi | \, dm \le m(C_n) \| \phi \|_{L^{\infty}_m} \le C m(C_n) \| \phi\|_{\alpha}$, whence $$|R_n \phi |_{\alpha} \le C ((c^{(n)})^{\alpha} + m(C_n)) \| \phi \|_{\alpha}$$  and similarly for $\|R_n \phi \|_{\alpha}$. Since $c^{(n)}$ et $m(C_n)$ decay exponentially fast by assumption, one obtains that $\|R_n\| = \mathcal{O} (\theta^n)$ for some $0 < \theta < 1$.
\end{proof}

\begin{lemma} \label{lemma:ren2}

$R(1)$ admits $1$ as a simple isolated eigenvalue, and the corresponding eigenprojector is given by $$\Pi \phi = \left( \int_Y \phi \, dm \right) \frac{h_Y}{\mu(Y)},$$ where $h_Y$ is the restriction to $Y$ of the density $h$ of the measure $\mu$ (and then $\frac{h_Y}{\mu(Y)}$ is the density of the absolutely continuous invariant probability for $\widehat{T}$).

Furthermore, we have $\Pi R'(1) \Pi = \frac{\Pi}{\mu(Y)}$, so that $\gamma$ in Proposition \ref{exp_renewal} is equal to $\frac{1}{\mu(Y)}$.

\end{lemma}

\begin{proof}

We note that $R(1) = P$ is the transfer operator of the Gibbs-Markov map $\widehat{T}$. Consequently, $1$ is a simple isolated eigenvalue, and the corresponding eigenprojector is given by the desired formula. 

We have $$\Pi R'(1) \Pi \phi = \left( \frac{\int_Y R'(1) h_Y \, dm}{ \mu(Y)} \right) \left( \frac{\int_Y \phi \, dm }{\mu(Y)} \right) h_Y,$$ whence $\gamma = \frac{\int_Y R'(1) h_Y \, dm }{\mu(Y)}$.

Since $R_n \phi = \mathbbm{1}_Y L^n(\mathbbm{1}_{C_n} \phi) = \mathbbm{1}_Y P(\mathbbm{1}_{C_n} \phi)$ for any function $\phi$, we have\begin{align*} 
\int_Y R'(1) h_Y \, dm & = \sum_n n \int_Y P(\mathbbm{1}_{C_n} h_Y) \, dm = \sum_n n \int_{C_n} h_Y \, dm\\
&= \sum_n n \mu(C_n) = \int_Y r \, d\mu = 1\end{align*}
by Kac's lemma, and we get $\gamma = \frac{1}{\mu(Y)}$.
\end{proof}

It remains to prove the aperiodicity condition: 

\begin{lemma} \label{lemma:ren3}

For all $z \in \bar{\D} \setminus \{1\}$, $I-R(z)$ is invertible on $V_\alpha(Y)$.

\end{lemma}

\begin{proof}

We first establish a Lasota-Yorke inequality for the operator $R(z)$. Remark that $$R(z)^k =\sum_{n_1, \ldots, n_k \ge 1} z^{n_1 + \ldots + n_k} R_{n_k} \circ \ldots \circ R_{n_1},$$ and that $$R_{n_k} \circ \ldots \circ R_{n_1} = \sum_{i_1 \in \mathcal{I}_{n_1}, \ldots, i_k \in \mathcal{I}_{n_k}}  M_{k, Q_{I_1, \ldots, I_k}},$$ where $Q_{I_1, \ldots, I_k} \in \mathcal{Q}_k$ is defined by $Q_{I_1, \ldots, I_k} = Y_{i_1} \cap \widehat{T}^{-1} Y_{i_2} \cap \ldots \cap \widehat{T}^{-(k-1)} Y_{i_k}$. 
Then, summing all the relations from Lemma \ref{lem:mnq} and noticing that $|z| \le 1$ and $n_1 + \ldots + n_k \ge k$, we have $\|R(z)^k \phi\|_{L^1_m} \le C |z|^k \| \phi \|_{L^1_m}$ and $|R(z)^k \phi|_{\alpha} \le C|z|^k \left( \theta^{-\alpha k} |\phi|_{\alpha } + \| \phi\|_{L^1_m}\right)$, arguing as in the proof of Lemma \ref{lem:LY}.

This shows that the spectral radius of $R(z)$ is less than $|z|$, while the essential spectral radius of $R(z)$ is strictly less than $1$ if $|z| = 1$, by Hennion's theorem \cite{Hennion}. Thus, the problem reduces to prove that the relation $R(z) \phi = \phi$, with $|z| = 1$ and $\phi \in V_{\alpha}(Y)$ implies that $z = 1$ or $\phi = 0$. 

Let $|z| = 1$ and $\phi \in V_{\alpha}(Y)$ non-zero satisfying $R(z) \phi = \phi$, that is $P(z^r \phi) = \phi$. By \cite[Proposition 1.1]{Mor}, we deduce that $\left( \frac{\phi}{h_Y} \right) \circ \widehat{T} = z^r  \frac{\phi}{h_Y}$. Since $(X,T,\mu)$ is mixing, and hence weakly mixing, by Proposition \ref{prop:aperiodicity} (see Appendix), we get that $z= 1$, concluding the proof.
\end{proof}

\begin{proof}[Proof of Theorem \ref{thm:decay_nue}] By lemmas \ref{lemma:ren1}, \ref{lemma:ren2} and \ref{lemma:ren3}, we can apply Proposition \ref{exp_renewal} and get $\|V_n  - \mu(Y) \Pi\| \le C \kappa^n$, i.e. $$\left\| V_n \phi - \left( \int_Y \phi \, dm \right) h_Y \right\|_{\alpha} \le C \kappa^n \|\phi \|_{\alpha},$$ for all $\phi \in V_\alpha(Y)$.

Let $\phi \in V_\alpha(Y)$ and $\psi \in L^1(\mu)$ supported in $Y$. We have 

$$ \int_X \phi  \, \psi \circ T^n \, dm = \int_X \mathbbm{1}_Y L^n(\mathbbm{1}_Y \phi) \psi \, dm = \int_Y (V_n \phi)  \, \psi \, dm,$$

Since \begin{align*}
\left|\int_Y V_n (\phi) \, \psi dm  - \int_Y\phi~dm\int_Y \psi~d\mu \right| &=  \left|\int_Y \left[V_n \phi - \left(\int_Y\phi \,dm\right)h_Y\right]\psi~dm \right| \\
&\le \left\|V_n \phi- \left(\int_Y \phi \, dm\right)h_Y\right\|_{\alpha}  \int_Y |\psi |~dm\\
&\le  C \kappa^n \|\phi\|_{\alpha} \|\psi\|_{L^1_m},
\end{align*}

we get \begin{align*}
\left|\int_X\phi \, \psi\circ T^n \,  dm -\int_Y\phi~dm\int_Y \psi~d\mu \right| & \le C\kappa^n \| \phi\|_{\alpha} \|\psi\|_{L^1_m} \\ 
&\le C \kappa^n \|\phi\|_{\alpha} \|\psi\|_{L^1_{\mu}},
\end{align*}

as $\|\psi\|_{L^1_m} \le \|h_Y^{-1}\|_{L^{\infty}_m} | \psi \|_{L^1_\mu} \le C \| \psi \|_{L^1_\mu}$, the density of $\mu$  being bounded from below on $Y$.

The theorem follows by taking $\phi h_Y$ for $\phi$, using the fact that $\| \phi h_Y \|_\alpha \le \|h_Y\|_\alpha \| \phi \|_\alpha \le C \| \phi \|_\alpha$.
\end{proof}

In order to apply Corollary \ref{cor:bc_nue} to the Rauzy-Veech-Zorich map, we need mixing, so we will rather consider the map $G = \T_1^2$ restricted to $\Delta_{\eps}$, $\eps = 0,1$, which admits $\tilde{\mu}_1 = 2 \mu_1 ( . \cap \Delta_\eps)$ as an invariant measure. If the good induced set $B$ is included in $\Delta_\eps$, then $\T_2 : B \to B$ is the first return map of $G$ to $B$, with associated return time $\tilde{n}_2 = \frac{n_2}{2}$.
It has been shown by Avila and Bufetov \cite{Avila_Bufetov} that the measure of the set $\{n_1 = n\}$ decays exponentially fast with $n$. To apply Corollary \ref{cor:bc_nue}, it remains to prove the condition on the Lipschitz constants:

\begin{lemma} The Lipschitz constant $c_i = c_{i, B_i}$ of $I_i : B \to B_i$ decays exponentially fast with $n$: there exist $0 < \gamma < 1$ and $C > 0$ such that $c_i \le C \gamma^n$ for all $n \ge 1$ and all $i \in \mathcal{I}_n$.

\end{lemma}

\begin{proof}
By Avila-Bufetov \cite{Avila_Bufetov}, $m(C_n)$ decays exponentially fast. The map $I_n: B \to Y_i$ is a composition of a linear
map $\lambda\to A\lambda$ followed by $A \lambda \to \frac{A\lambda}{|A\lambda |_1}$.  $A$ is a non-negative matrix
and  $\frac{\lambda_i}{\lambda_j}$ is bounded for all $\lambda=(\lambda_1,\ldots,\lambda_d)$ in $B$. Hence $1\ge \frac{|A\lambda |}{\|A\|}>C>0$ for
all $\lambda \in B$ (this is an observation of Avila and Bufetov~\cite[Page 9]{Avila_Bufetov}). Furthermore $\frac{|A\lambda^{'}|_1}{|A\lambda|_1}<C$
for all $\lambda,\lambda^{'}$ in $B$ by Proposition 1.3. Thus the exponential decay of volume implies that at least one direction contracts 
exponentially under $I_n$ by a factor $\gamma^{1/d}$ and hence all directions do, this implies $L_n \le C (\gamma^{\frac{1}{d}})^n$.
\end{proof}

We can then conclude:
\begin{theorem} 

Suppose $\{\phi_n\}$ is  a sequence of positive functions with support  in $B$, bounded in $V_\alpha(B)$
with $E_n:=\sum_{j=1}^{n}  \tilde{\mu}_1 (\phi_j)$  divergent. Then
\[
\frac{1}{E_n} \sum_{j=1}^n \phi_j \circ G^j (x) \to 1
\]
for $\mu_1$ a.e.\ $x\in \Delta_\eps$.

\end{theorem}

This theorem applies in particular to sequences of characteristic functions of balls included in $B$.

\section{Extreme Value Laws for $\T_1$ and $\T_2$.}\label{extreme}


By expressing $\T_2$ as a multidimensional piecewise expanding map with exponential decay of correlations with respect to a quasi-H\"older norm 
versus $L^1$ we are able to apply results on Extreme Value statistics for such systems. Let $\phi: B \to \R \cup \{+\infty\}$ be a function, 
strictly maximized at a point $p_0\in B$, 
which is sufficiently regular that for large $u$ the set $\{ x \in B: \phi (x) >u\}$ corresponds to a topological ball centered at $p_0$. 
Let $$M_n (x):= \max \{ \phi (x), \phi \circ \T_2 (x), \ldots, \phi \circ \T_2^n (x)\}.$$
The aim is to show that we have a non-degenerate limit law for $M_n$, which we think of as a random variable.  Since almost surely $M_n$ converges to  $\phi(p_0)$, since $\mu_2$ is ergodic, for such a law, we need to rescale our variable.  To this end, for each $t$ we define scaling constants $u_n(t)$
by $n\mu_2 (\phi > u_n(t) )\to t$.  For example, if $\phi(x)=-\log d(x,p_0)$ then $u_n(t)= d^{-1}[\log C(d)+ \log n -\log t]$ where $C(d)$ is the constant giving the volume of the  unit ball in $d$ dimensional Euclidean space (if $d$ is the dimension of $B$). In fact we may always write $u_n(t)$ in the form $$u_n (t)=u_n^{\T_2} (t)=\frac{g(t)}{a_n}+b_n$$ for some function $g(t)$ and sequence of constants $a_n$, $b_n$.  In our example $a_n=d$, $g(t)=\log C(d) -\log t$ and $b_n=\frac{1}{d} \log n$.
where $d$ is the dimension of $B$.  We say that we have an \emph{Extreme Value Law} if the variable
$M_n$ under scaling by $u_n$ converges to some non-degenerate distribution.  For the classical application of these ideas to i.i.d. processes, see \cite{LLR83}.  For more recent applications to dynamical systems, as we have here, see for example \cite{Collet, FreFreTod10, HNT}.

There is a close connection between rare events point processes (REPP), extreme value laws and hitting times.  First we describe what we mean by a compound Poisson process. Let $\mathcal{R}$ be the ring of subsets of $\R^+$ generated by the semi-ring of subsets of form $[a,b)$ so that an 
element of $J\in \mathcal{R}$ has the form $J=\cup_{i=1}^n [a_i,b_i)$.

\begin{definition} 
Let $X_1$, $X_2$, $\ldots$, be an iid sequence of random variables with common exponential distribution of mean $\frac{1}{\theta}$. Let $D_1$, $D_2$, $\ldots$
be another iid sequence of random variables, independent of $X_i$ and with distribution function $\eta$. We say that $N$ is a \emph{compound Poisson process of intensity $\theta$ and multiplicity distribution function $\eta$} if for every $J\in \mathcal{R}$
\[
N(J)=\int1_{J} d\left(\sum_{i=1}^{\infty} D_i \delta_{X_1+\ldots +X_i}\right),
\]
where $\delta_t$ is  the Dirac measure at $t$.  If $P(D_1=1)=1$ then $N$ is  the standard Poisson distribution and for every $t>0$ the random variable
$N([0,t))$ has a Poisson distribution of mean $\theta t$.

\end{definition}

\begin{remark} 
In our applications $\eta$ will follow a geometric distribution of parameter $\theta \in (0,1]$ and $\pi (k):=P(D_1=k) =\theta (1-\theta)^k$ for every 
integer $k\ge 0$. In this case the random variable follows a P\'olya-Aeppli distribution, 
\[
P(N([0,t))=k)=e^{-\theta t} \sum_{i=1}^k \theta^i (1-\theta)^{k-i} \frac{(\theta t)^i}{i!}  \left( \begin{array}{c}
k-1 \\
i-1 \end{array} \right).
		\] 
		\end{remark}

		Define $v_n^{\T_2} (t) :=\mu_2 (\phi > u_n^{\T_2})^{-1}$ so that $v_n^{\T_2} (t)\sim \frac{n}{t}$. If $J=\cup_{i=1}^n [a_i,b_i) \in \mathcal{R}$ and $\gamma>0$, define
		$\gamma J=\cup_{i=1}^n [\gamma a_i,\gamma b_i) \in \mathcal{R}$.
		
		We define the rescaled REPP $N_n^{\T_2}$ as
		\begin{equation}
		N_n^{\T_2} (J):= \sum_{j\in v_n^{\T_2} J \cap \N_0} 1_{(\phi\circ\T_2^j > u_n^{\T_2})}.
		\label{eq:N_n T_2}
		\end{equation}
		EVLs and limit laws  for $N_n^{\T_2}$ for $\T_2$  follow directly from~\cite[Proposition 3.3]{AytFreVai13}.  We state them here:
		
		\begin{proposition} Suppose that $p_0$ satisfies the Keane condition.
		(1) If $p_0$ is not a periodic point for $\T_2$ then $\mu_2 \{ M_n \le u_n (t) \}\to e^{-t}$ and the REPP $N_n^{\T_2}$ converges in distribution to a standard 
		Poisson process $N$ of intensity $1$.
		
		(2) If $p_0$ is a  repelling periodic point of prime period $k$ then $\mu_2 \{ M_n \le u_n (t) \}\to e^{-\theta t}$ where $\theta=1-|Jac(D\T_2^{-k}) (p_0)|$
		and the REPP $N_n^{\T_2}$ converges in distribution to  a compound Poisson process $N$ with intensity $\theta$ and multiplicity distribution function
		$\eta$ given by $\eta (j)=\theta(1-\theta)^j$ for all integers $j\ge 0$.
		\label{prop:BC for T2}
		\end{proposition}
		
		Now define $u_n^{\T_1}(t)$ to be so that $n\mu_1 (\phi > u_n^{\T_1})\to t$ as $n\to \infty$.  Then setting  $v_n^{\T_1}(t) :=\mu_1 (\phi > u_n^{\T_1}(t))^{-1}$, we can define the REPP $N_n^{\T_1}$ by changing all the appearances of $\T_2$ in \eqref{eq:N_n T_2} to $\T_1$.  We then have the following corollary.
		
		\begin{corollary} Suppose that $p_0$ satisfies the Keane condition.
		(1) If $p_0$ is not a periodic point for $\T_1$ then $\mu_1\{ M_n \le u_n^{\T_1} (t) \}\to e^{-t}$ and the REPP $N_n^{\T_1}$ converges in distribution to a standard 
		Poisson process $N$ of intensity $1$.
		
		(2) If $p_0$ is a  repelling periodic point of prime period $k$ then $\mu_1 \{ M_n \le u_n (t) \}\to e^{-\theta t}$ where $\theta=1-|Jac(D\T_1^{-k}) (p_0)|$
		and the REPP $N_n^{\T_1}$ converges in distribution to  a compound Poisson process $N$ with intensity $\theta$ and multiplicity distribution function
		$\eta$ given by $\eta (j)=\theta(1-\theta)^j$ for all integers $j\ge 0$.
\end{corollary}

Observe that $p_0$ as above, that is the point where $\phi$ takes its maximum, we can choose our set $B$ to contain $p_0$, so that the result in Proposition~\ref{prop:BC for T2} applies to the corresponding first return map $\T_2$.
The proof that we can always pass from the result on the first return map (i.e., $\T_2$ here) to the original case (i.e., for $\T_1$), which is a simple generalisation of the main result in \cite{HayWinZwe13}, appears in \cite{FreFreTodVai15}.  Note that the second part was already proved in \cite{FreFreTod13}.

\section{Return and hitting time statistics.}

In this section we consider a natural notion of recurrence which, as in \cite{FreFreTod10}, is analogous to the EVL perspective in the previous section.
Suppose $p_0\in B$ and $U_n$ is a sequence of balls nested at $p_0$. Let $\tau_{2, U} (x):= \min \{n\ge 1: \T_2^n (x)\in U\}$.
 We say that $\T_2$ has \emph{hitting time statistics to $\{U_n\}$ with distribution $H(t)$}
if
\[
\lim_{n\to \infty} \mu_2 \left( x\in B: \tau_{2, U_n} (x) \le \frac{t}{\mu_2 (U_n))} \right) = H(t).
\]
We say that $\T_2$ has \emph{return  time statistics  to $\{U_n\}$ with distribution $\tilde{H}(t)$}
if
\[
\lim_{n\to \infty} \frac{1}{\mu_2 (U_n)}\mu_2 \left( x\in U_n: \tau_{2, U_n} (x) \le \frac{t}{\mu_2 (U_n))} \right) =  H(t).
\]
There is a large body of literature on this topic: we refer the reader to \cite{AbaGal01, HLV} and references therein for further information on this notion of asymptotic recurrence.

As in \cite{FreFreTod10} sets of the form $\{x\in B: M_n \le u_n (t) \}$ can be rewritten as $\{x\in B: \tau_{2, U_n} (x) \le \frac{t}{\mu_2 (U_n)}\}$, hence the basic part of Proposition~\ref{prop:BC for T2} can be written:

	\begin{proposition} Suppose that $p_0$ satisfies the Keane condition.
		(1) If $p_0$ is not a periodic point for $\T_2$ then 
		\[
\lim_{n\to \infty} \mu_2\left( x\in U_n: \tau_{2, U_n} (x) \le \frac{t}{\mu_2 (U_n))} \right)= 1-e^{- t}.
\] 
		
		(2) If $p_0$ is a  repelling periodic point of prime period $k$ then
		 \[
\lim_{n\to \infty} \mu_2\left( x\in U_n: \tau_{2, U_n} (x) \le \frac{t}{\mu_2 (U_n))} \right)= 1-e^{-\theta t}
\] 
where $\theta=1-|Jac(D\T_2^{-k}) (p_0)|$
			\label{prop:HTS for T2}
		\end{proposition}
For typical points this was originally proved in \cite[Theorem 2.1]{BruSauTroVai03} and in the periodic case this follows by \cite[Corollary 4]{FreFreTod13}, but the full dichotomy, covering \emph{all} points, comes from \cite{AytFreVai13}.  To convert the results of this proposition from hitting time statistics to return time statistics, we use the main result of \cite{HLV} which shows that these limits then become $(1-\theta)+\theta(1-e^{-\theta t})$ (where we take $\theta=1$ in the non-periodic case, so nothing changes).

To convert the laws for $\T_2$ to $\T_1$ we use \cite[Theorem 10.3]{HayWinZwe13}.  To set up some of the notation here, we suppose that $r_U$ is the first return time to $U$ for the original dynamics, and $r_{Y, U}$ is the first return time for the speeded up (first return map) dynamics on $Y$ and $\mu_U=\mu|_U/\mu(U)$.

\begin{theorem}
Let $(X, T \mu)$ be an ergodic probability-preserving system and $Y$ be a measurable set with $\mu(Y)>0$.  Assume that $(H_\ell)_\ell$ is a sequence of measurable sets in $Y$ with $\mu(H_\ell)\to 0$ as $\ell\to \infty$ and that $\tilde R$ is a any random variable with values in $[0, \infty]$.  Then
$$\mu_Y(H_\ell) r_{Y,  H_\ell} \stackrel{\mu_{H_\ell}}{\Longrightarrow} \tilde R \text{ as } \ell \to \infty$$
iff
$$\mu(H_\ell) r_{H_\ell} \stackrel{\mu_{H_\ell}}{\Longrightarrow} \tilde R \text{ as } \ell \to \infty.$$
\end{theorem}

Here $\stackrel{\mu_{H_\ell}}{\Longrightarrow}$ means pointwise convergence at all continuity points of $\tilde R$ where the LHS is considered w.r.t. $\mu_{H_\ell}$.
So for $\tau_{1, U} (x):= \min \{n\ge 1: \T_1^n (x)\in U\}$, the above result can be interpreted with $\tau_{2, U}$ on $H_\ell$ being $r_{Y, H_\ell}$ and $\tau_{1, U}$ on $H_\ell$ being $r_{H_\ell}$, to obtain the statement of Proposition~\ref{prop:HTS for T2} for $\T_1, \mu_1$.  As in the previous section, we may assume that the domain $B$ contains $p_0$.

\section{The Teichm\"uller flow on the space of translation surfaces}
\label{sec:flow}

In this section we relate the dynamical structures we described in Section~\ref{sec:background} to the Teichm\"uller flow on the space of translation surfaces.  We do
not present  any new results in this section.   We will first introduce invertible versions $\mathcal{R}_0, \mathcal{R}_1$ and $\mathcal{R}_2$ of the maps presented in Section~\ref{sec:background}. The key fact we use is that these maps are first return maps for the flow to adapted cross sections, and give a clearer relation to the translation surfaces, which are represented as points in their phasespace.

\subsection{Translation surfaces: the zippered rectangle construction}
Given an irreducible pair $\pi=(\pi_0, \pi_1)$ and a length vector $\lambda\in \R_+^{\A}$, let $T_\pi^+$ denote the subset of vectors $\tau=(\tau_\aalpha)_{\aalpha\in \A}\in \R^{\A}$ such that 
$$\sum_{\pi_0(\aalpha)\le k}\tau_\aalpha>0 \text{ and } \sum_{\pi_1(\aalpha)\le k}\tau_\alpha<0$$
for $1\le k\le d-1$.  We say that $\tau$ has \emph{type 0} if the total sum $\sum_{\aalpha\in \A}\tau_\aalpha$ is positive and \emph{type 1} if the total sum is negative.  

Next we will use the matrices $\M$ and intervals $I_\aalpha^{\pi_\eps}$ defined in Section~\ref{ssec:IET}.  Then given $\pi$ and $\tau\in T_\pi^+$ we define the height data by $h:=-\M \tau$.  One can check that $\tau\in T_\pi^+$ implies that each element $h_\aalpha$ for $\aalpha\in \A$ is strictly positive.  Now given $(\pi, \lambda, \tau)$, for each $a\in \A$ we can define the rectangles $R_\aalpha^{\pi_0}=I_\aalpha^{\pi_0}\times [0, h_\aalpha] \subset \R^2$ and $R_\aalpha^{\pi_1}=I_\aalpha^{\pi_1}\times [0, -h_\aalpha]\subset \R^2$.  We can then form the \emph{translation surface} $M=M(\pi, \lambda, \tau)$ by identifying the top of each rectangle $R_\aalpha^{\pi_0}$ with the bottom of the corresponding rectangle $R_\aalpha^{\pi_1}$ and then `zipping up' by making a natural identification of pairs of protruding sides of the rectangles: for more details see \cite[Chapter 2.7]{Via08}, \cite{Yoc07}. The area of $M(\pi, \lambda, \tau)$ can be defined as $\mbox{area}(\pi, \lambda, \tau) := \lambda \cdot h = \sum_{\aalpha\in \A}\lambda_\aalpha h_\aalpha$.  The structure here can be thought of as a Riemann surface with a non-zero holomorphic 1-form or equivalently, as a flat Riemannian metric on a surface with finitely many singularities of conical type and a parallel unit vector field. 
 
 Note that the underlying IET here is a first return map of the vertical flow on the translation surface to the interval $[0, \sum_{\aalpha\in \A}\lambda_\aalpha]$.

Fix $\rauz$ a Rauzy class.  Let 
$$\hat\H=\hat \H(\rauz):=\left\{(\pi, \lambda, \tau)\in \rauz\times\R_+^{\A}\times T_\pi^+\right\}.$$  We extend the Rauzy-Veech induction map $\hat\T_0$ to a map $\hat{\mathcal{R}}_0$ on $\hat\H$ by $\hat{\mathcal{R}}_0(\pi, \lambda, \tau)=(\pi', \lambda', \tau')$, where $(\pi', \lambda')=\hat\T_0(\pi, \lambda)$ and $\tau'=\Theta^{-1 *}(\tau)$ (recall the description of $\Theta$ given in Remark~\ref{rmk:Theta}). The height data $h'$ of $(\pi', \lambda', \tau')$ can be expressed as $h' = \Theta(h)$.  Moreover, setting 
$$\R_{\pi,\eps}^{\A}:=\{\lambda\in \R_+^{\A}: (\pi, \lambda) \text{ has type } \eps\} \text{ and } T_{\pi,\eps}:=\{\tau\in T_\pi^+: \tau \text{ has type } \eps\},$$
it can be shown (see eg \cite[Chapter 2.7]{Via08}) that:

\begin{proposition}
\begin{enumerate}[label=({\alph*}),  itemsep=0.0mm, topsep=0.0mm, leftmargin=7mm]
\item $\Theta^{-1 *}$ sends $T_\pi^+$ injectively inside $T_{\pi'}^+$.
\item (Markov) $ \hat{\mathcal{R}}_0(\{\pi\}\times\R_{\pi, \eps}^{\A}\times T_{\pi}^+)=\{\pi'\}\times\R_{+}^{\A}\times T_{\pi', 1-\eps}$.
\item Every $(\pi', \lambda', \tau')$ such that $\sum_{\alpha\in \A}\tau_\alpha'\neq 0$ has a unique preimage by $\hat{\mathcal{R}}_0$.
\item If $\hat{\mathcal{R}}_0(\pi, \lambda, \tau)=(\pi', \lambda', \tau')$ then the areas of $M(\pi, \lambda, \tau)$ and $M(\pi', \lambda', \tau')$ are equal. 
\end{enumerate}
\label{prop:T0 on surf}
\end{proposition}


\subsection{Teichm\"uller flow}

The Teichm\"uller flow on $\hat\H$ is defined as the induced action $\T=(\T^t)_{t\in \R}:\hat\H\to \hat\H$ of the diagonal subgroup
\begin{equation*}
\begin{pmatrix} e^t&0 \\
0 & e^{-t} \\
\end{pmatrix} \text{ for } t\in \R,
\end{equation*}
given by $\T^t(\pi, \lambda, \tau)=(\pi, e^t\lambda, e^{-t}\tau)$.
For $c>0$ we define$$\H_c:=\{(\pi, \lambda, \tau)\in \hat\H:|\lambda|=c\}.$$
The trajectory of a point in $\hat\H$ hits $\H_c$ precisely once.  
We are looking for transformations from $\H_c$ back to itself of the form $\hat{\mathcal{R}}_0\circ \T^t$ for some $t$.  Noticing that if $(\pi', \lambda')=\hat{\mathcal{R}}_0(\pi,\lambda)$ and $(\pi, \lambda)$ is of type $\eps$, then $|\lambda'|=|\lambda|\left(1-\frac{\lambda_{\aalpha(1-\eps)}}{|\lambda|}\right)$, we see that the relevant time $t$ is
$$r_0=r_0(\pi, \lambda):=-\log \left(1-\frac{\lambda_{\aalpha(1-\eps)}}{|\lambda|}\right) \text{ where } (\pi, \lambda) \text{ is of type } \eps.$$
That is to say, we are interested in the map from $\H_c$ to itself given by
$$\mathcal{R}_0=\hat{\mathcal{R}}_0\circ \T^{r_0}:(\pi, \lambda, \tau)\mapsto \hat{\mathcal{R}}_0(\pi, e^{r_0}\lambda, e^{-r_0}\tau).$$
From now on we restrict ourselves to $$\H=\H_1.$$  Then we observe that the map above can actually be interpreted as an extension of the Rauzy-Veech \emph{renormalisation} map $\T_0$ since $\mathcal{R}_0 (\pi, \lambda, \tau)=(\pi', \lambda'', \tau'') = (\T_0(\pi, \lambda), \tau'')$ where
$$(\pi', \lambda', \tau')=\hat{\mathcal{R}}_0(\pi, \lambda, \tau), \quad \lambda''=\frac{\lambda'}{1-\lambda_{\aalpha(1-\eps)}}, \quad \tau''=\tau'(1-\lambda_{\aalpha(1-\eps)}).$$
 The next result is \cite[Corollary 2.24]{Via08} and \cite[Lemma 4.3]{Via08}.

\begin{proposition}
$\mathcal{R}_0:\H\to \H$ is an (almost everywhere) invertible Markov map and preserves the area of the corresponding translation surfaces. The standard volume form $m_{\H} = d \pi d \lambda_1 d \tau$, where $d \lambda_1$ is the Lebesgue measure induced on $\Delta_{\mathcal{A}}$ and $d \tau$ is the Lebesgue measure on $T_\pi^+$, is invariant under $\mathcal{R}_0$.
\end{proposition}

From now on, we will only consider translation surfaces of area $1$, i.e. elements of the set $$\hat\H_{(1)} := \{(\pi, \lambda, \tau) \in \hat\H : \mbox{area}(\pi, \lambda, \tau) = 1\}.$$
This set is invariant under both the Teichm\"uller flow $\T = (\T^t)_{t \in \R}$ and the invertible Rauzy-Veech induction $\hat{\mathcal{R}}_0$. We also set $\H_{(1)} := \hat\H_{(1)} \cap \H$, which is invariant under the invertible Rauzy-Veech renormalization map $\mathcal{R}_0$.

We consider the pre-stratum obtained as the quotient of the fundamental domain $\{(\pi, \lambda, \tau) \in \hat\H_{(1)} : 0 \le \log | \lambda| \le r_0(\pi, \lambda)\}$ by the equivalence relation $$\T^{r_0(\pi, \lambda)}(\pi, \lambda, \tau) \sim \mathcal{R}_0(\pi, \lambda, \tau) \text{ for all } (\pi, \lambda, \tau) \in \H_{(1)}.$$
Since $\mathcal{R}_0$ commutes with the flow, the latter induces a flow $\T = (\T^t)_{t \in \mathbb{R}}$ on the pre-stratum, that we also call Teichm\"uller flow. 

The map $\mathcal{R}_0 : \H_{(1)} \to \H_{(1)}$ is then naturally identified with the Poincar\'e return map of this flow to the cross section $\H_{(1)}$. The volume form $m_\H$ induces a volume form $m_{\H_{(1)}}$ on $\H_{(1)}$ which is still invariant under $\mathcal{R}_0$. The key fact is that $m_{\H_{(1)}}$ gives finite mass to $\H_{(1)}$, a fact which was demonstrated by Veech \cite{Vee82}.

\subsection{Recoded Teichm\"uller flow and inducing}

The moves described above mean that $\mathcal{R}_0$ can now be interpreted as the first return map of the Teichm\"uller flow to $\H_{(1)}$, and indeed it is convenient for us to redefine the flow as a suspension flow which is locally defined by $\T^t(\pi, \lambda, \tau, s)=(\pi, \lambda, \tau, t+s)$ on the space
$$\H^{r_0}_{(1)}:=\left\{(\pi, \lambda, \tau, s)\in \H_{(1)} \times \R: 0\le s\le r_0(\pi, \lambda)\right\}/\sim$$
where $(\pi, \lambda, \tau, r_0(\pi, \lambda))\sim (\pi', \lambda'', \tau'', 0)$ and $\mathcal{R}_0(\pi, \lambda, \tau)=(\pi', \lambda'', \tau'')$.  We refer to $r_0$ as the \emph{roof function} for this suspension flow.

A key fact in Proposition~\ref{prop:T0 on surf}(b) is that given $(\pi, \lambda, \tau)\in \H_{(1)}$, if $(\pi, \lambda)$ is of type $\eps$, then $\tau'$ is of type $1-\eps$.  So if the first $k$ iterates $(\pi^j, \lambda^j, \tau^j)$ for $j=1,\ldots, k$  of $\mathcal{R}_0$ do not change the type of $(\pi^j, \lambda^j)$, then the types of $(\pi^j, \lambda^j)$ and $\tau^j$ are different ($\eps$ and $1-\eps$) for $j\in \{1, \ldots, k\}$.  So the first time $k$ that the types of $(\pi^k, \lambda^k)$ and $\tau^k$ are the same is the first time that $(\pi^k, \lambda^k)$ changes type.  That is, exactly $n_1(\pi, \lambda)$.  Therefore, setting $\zoret:=\zoret_0\cup \zoret_1$, where for $\eps\in \{0, 1\}$,
$$\zoret_\eps:=\left\{(\pi, \lambda, \tau)\in \H_{(1)}:(\pi, \lambda) \text{ and } \tau \text{ both have type } \eps\right\},$$
 we define $\mathcal{R}_1:\zoret\to\zoret$ as the first return map by $\mathcal{R}_0$ to $\zoret$.
(We can do this with $\hat{\mathcal{R}}_1$ on $\hat\H$ too.)  This map can be seen as an extension of the Rauzy-Veech-Zorich renormalisation map for the same reasons as for $\mathcal{R}_0$: if $\mathcal{R}_1(\pi, \lambda, \tau) = (\pi', \lambda', \tau')$, then $\T_1(\pi, \lambda) = (\pi', \lambda')$. Thus we can produce a new description of our Teichm\"uller flow.  

We omit the description of this since we go straight to the description given by taking an adapted induced set $B_{\H_{(1)}}\subset \zoret$ and the first return map $\mathcal{R}_2$ to $B_{\H_{(1)}}$ by $\T$. This map will also be the first return map of $\mathcal{R}_0$ to $B_{\H_{(1)}}$.
The choice of $B$ in Section \ref{ssec:MorPol} was made in order to ensure uniform expansion for the first return map. Since we are now dealing with an invertible map, we will also need uniform contraction in the stable direction. We follow the construction of \cite{Avila_Gouezel_Yoccoz}, and choose a good set $B$, which is the image of an inverse branch of $\T_0$. We refer to \cite[Section 4.1.3]{Avila_Gouezel_Yoccoz} for the precise definition of $B$. This set can be written as $B = \{\pi\} \times \{ \frac{\Theta^{\star} \lambda }{|\Theta^{\star} \lambda|} \, : \, \lambda \in \Delta_{\A}\}$, where $\Theta$ is a finite product of the matrices mentionned in Remark \ref{rmk:Theta}.

We then set $B_{\H_{(1)}} = (B  \times T_B^{+}) \cap \H_{(1)}$, where $T_B^{+}$ is defined by the relation $\Theta^{\star} T_B^{+} = T_{\pi}$, and we consider the first return map $\mathcal{R}_2$ of $\mathcal{R}_0$ to $B_{\H_{(1)}}$. This map can be written as a skew product over the first return map $\T_2$ of $\T_0$ to the set $B$, i.e. $\mathcal{R}_2(\pi, \lambda, \tau) = (\pi', \lambda', \tau')$, where $(\pi', \lambda') = \T_2(\pi, \lambda)$, and $\tau'$ depends on $\pi, \lambda$ and $\tau$.

The map $\mathcal{R}_2$ preserves the renormalised restriction $m_{B_{\H_{(1)}}}$ of $m_{\H_{(1)}}$ to $B_{\H_{(1)}}$. By \cite[Lemma 4.3]{Avila_Gouezel_Yoccoz}, this map is a hyperbolic skew product over the uniformly expanding Markov map $\T_2$, in the sense of \cite[Definition 2.5]{Avila_Gouezel_Yoccoz}, and henceforth it admits exponential decay of correlations for Lipschitz observables: there exists $C>0$ and $0 < \alpha < 1$ such that $$\left| \int \phi \, \psi \circ \mathcal{R}_2^n d m_{B_{\H_{(1)}}} - \int \phi \, d m_{B_{\H_{(1)}}} \int \psi \, d m_{B_{\H_{(1)}}}  \right| \le C \alpha^n \| \phi \|_{\rm Lip} \| \psi\|_{\rm Lip},$$ for all $\phi, \psi \in {\rm Lip}$, see Young \cite{Young}. 

Since $\mathcal{R}_0$ is the Poincar\'e return map of the flow $\T$ to the section $\H_{(1)}$, the map $\mathcal{R}_2$ is itself the Poincar\'e return map of $\T$ to the section $B_{\H_{(1)}}$. This gives a roof function $r_2 : B_{\H_{(1)}} \to \mathbb{R}_+$ defined almost everywhere. Clearly, the roof function depends only on $(\pi, \lambda)$, so we can reduce it to a roof function $r_2 : B \to \mathbb{R}_+$. We define the suspension 
$$B_{\H_{(1)}}^{r_2}:=\left\{(\pi, \lambda, \tau, s)\in B_{\H_{(1)}}\times \R: 0\le s\le r_2(\pi, \lambda)\right\}/\sim$$
where $(\pi, \lambda, \tau, r_2(\pi, \lambda))\sim (\pi', \lambda'', \tau'', 0)$ and $\mathcal{R}_2(\pi, \lambda, \tau)=(\pi', \lambda'', \tau'')$. Again, we can redefine the flow $\T$ as a suspension flow on $B_{\H_{(1)}}^{r_2}$ given by $\T^t(\pi, \lambda, \tau, s) = (\pi, \lambda, \tau, t+s)$, which preserves the measure $\mu_\T = \frac{(m_{B_{\H_{(1)}}} \times m)|_{B_{\H_{(1)}}^{r_2}}}{(m_{B_{\H_{(1)}}} \times m)(B_{\H_{(1)}}^{r_2})}$ where $m$ is the Lebesgue measure on $\R$.

We now revert to  a form  which matches Pollicott's~\cite{Pollicott}  notes as well as corresponds to our sections above. Since the roof function depends only on $(\pi, \lambda)$, we can project  into a semi-flow by removing the $\tau$ parameter: then the actual flow can be reconstructed as the natural extension of what we have produced.  Namely, we let
$$B^{r_2}:=\left\{(\pi, \lambda, s)\in B\times \R: 0\le s\le r_2(\pi, \lambda)\right\}/\sim$$
where $(\pi, \lambda, r_2(\pi, \lambda))\sim (\pi', \lambda'', 0)$ and $\T_2(\pi, \lambda)=(\pi', \lambda'')$.  Clearly $\T_2$ is still a first return map to $B$.
Later we will simplify notation further and write simply $x=(\pi, \lambda)$.

The notation we use for the semi-flow is $\F_t:B^{r_2} \to B^{r_2}$, defined locally by $\F_t (x, u)=(x, u+t)$, with the relevant identifications i.e.
 $(x, r_2(\pi, \lambda))\sim (\T_2 (x), 0)$.  

The semi-flow $\F = \{ \F_t\}_{t \in \R}$ preserves the acip $ \mu_\F$ given by $$\mu_\F =\frac{(\mu_2\times m)|_{B^{r_2}}}{(\mu_2\times m)(B^{r_2})}=\frac{(\mu_2\times m)|_{B^{r_2}}}{\int r_2~d\mu_2},$$ where $\mu_2$ is the acip for $\T_2$ and $m$ is the Lebesgue measure on $\R$.
 
 \begin{remark}
Since $\T_2$ is a first return map for $\T_0$, which in turn is a first return map for our Teichm\"uller semi-flow, any small ball in $B^{r_2}$ is isomorphic to the corresponding ball in $B^{r_0}$ ($B^{r_0}$ being defined similarly to $B^{r_2}$, with $r_0$ as the roof function).  More precisely, this is true if our ball is contained in a strip $\{(x, t): x\in B_k, 0\le t\le r_2(x)\}$ for some $k$. Recall that $\mathcal{Q} = \{B_i\}_{i \in \mathcal{I}}$ is the natural partition of the map $\T_2$ defined in \ref{ssec:MorPol}.
\label{rmk:Br2 isom}
\end{remark}
 

\section{Statistical properties of the Teichm\"uller flow}
\label{sec:TF stat}

In this section we extend our Borel-Cantelli Lemmas and EVLs to the Teichm\"{u}ller flow.

\subsection{Borel-Cantelli Lemmas for the semi-flow}

Here we will use ideas from the proof of \cite[Theorem 2]{GNO}, primarily Step 1 of that proof.  The main (obvious) difference is that we are dealing with continuous time.

Given a family of sets $U=(U_s)_{s\ge 0}$ set $\psi = (\psi_s)_{s\ge 0}$ where $\psi_s:=\mathbbm{1}_{U_s}$ and  $E_t(U)=E_t(\psi) = \int_0^t \left( \int \psi_s d \mu_\F \right) ds$. We say that $U$ is a family of \emph{shrinking sets} if $s_1<s_2$ implies $U_{s_2}\subset U_{s_1}$.  In this section we will prove that if $U=(U_s)_{s\ge 0}$ is  a family of shrinking sets with some monotonicity  condition and $\lim_{t\to \infty}E_t(U)=\infty$ then
\[
\lim_{t\to \infty}\frac1{E_t(U)}\int_0^t \mathbbm{1}_{U_s}\circ \F_s (x, u)~ds =1 \quad\text{ for } \mu_\F\text{-a.e. } (x, u)\in B^{r_2}.
\]
This result is contained in Theorem~\ref{prop:restrict T2 flow}; in particular, the smoothness condition is given there.  We prove in the following subsection that this condition is indeed satisfied for a natural family of sets, namely nested balls.



Recall that $B$ is partitioned (almost everywhere) into sets $\{B_k\}_k$.   For $i\in \N_0$, define
$$B_k^i:=\Big\{(x, t)\in B_k\times \R_+:i\le t<\min\{i+1, r_2(x)\}\Big\}.$$  
So we can write $B^{r_2}=\cup_k\cup_iB_k^i$ almost everywhere.
We will restrict our Borel-Cantelli Lemmas to these sets $B_k^i$, which will be sufficient to prove the general case.  Indeed, we define the restricted indicator function
\[
\psi_{B_k^i, s} := \mathbbm{1}_{U_s\cap B_k^i}
\]
and first study the recurrence properties of the family $\psi_{B_k^i}=(\psi_{B_k^i, s})_{s\ge 0}$.  We do this by inducing, for which we need the right time scale.  
Since $\mu_\F$ is ergodic and $\int r_2~d\mu_2<\infty$, we immediately obtain the following lemma where $\overline{r_2}:=\int r_2~d\mu_2$.

\begin{lemma}
For each $\eps>0$ there exists $T\ge 0$ and a set $X_{\eps, T}\subset B^{r_2}$ such that $(x, u)\in X_{\eps, T}$ and $t\ge T$ implies
$$\left|\frac{t}{ \#\{s\in [0, t): \F_s (x, u)\in B\}}-\overline{r_2}\right|<\eps.$$
Moreover, $\mu_\F(X_{\eps, T})\to 1$ as $T\to\infty$. \label{lem:av flow time}
\end{lemma} 

Now, for each $\eps\in \R$, we define the induced function on $x\in B$ 
\begin{equation}\overline\psi_{n, B_k^i, \eps}(x):=\int_0^{r_2(x)}\left(\mathbbm{1}_{U_{n(\overline{r_2}+\eps) + s}}\cdot \mathbbm{1}_{B_k^i}\right)\circ \F_s (x, 0)~ds,
\label{eq:overline psi}
\end{equation}
and denote the family as $\overline\psi_{B_k^i, \eps}=(\overline\psi_{n, B_k^i, \eps})_n$. Note that $\int \overline\psi_{n, B_k^i, \eps}(x)d\mu_2 \le  \overline{r_2} \mu_\F (U_{n(\overline{r_2}+\eps)} \cap B_k^i)$ as $\mu_\F=\frac{1}{\overline{r_2}}(\mu_2\times m)|_{B^{r_2}}$.
We will be able to compare the long-term behaviour of this function with different values of $\eps$, and compare them all to the long-term behaviour of the flow.
This is necessary as we sample at discrete times, and the nested balls are shrinking in continuous time. 

We will use  the following lemma, which is \cite[Lemma 4.2]{GNO}.

\begin{lemma}
Suppose that $g:\R_+\to \R_+$ is decreasing and $\sum_{i=0}^\infty g(i)=\infty$. Then,
\begin{enumerate}[label=({\alph*}),  itemsep=0.0mm, topsep=0.0mm, leftmargin=7mm]
\item For all $\eps>0$ and all $n \ge 0$, 
$$\frac{\int_0^{(1+\eps)n}g(t)~dt}{\int_0^ng(t)~dt}\le 1+\eps.$$
\item $$\lim_{n\to \infty}\frac{\int_0^ng(t)~dt}{\sum_{j=0}^{n-1}g(j)}=1.$$
\end{enumerate}
\label{lem:decr seq}
\end{lemma}

\begin{theorem}
Suppose that $B_k^i$ is such that
\[
\lim_{t \to \infty} \int_0^t \mu_\F(U_s\cap B_k^i) ~ ds =\infty,
\]
i.e.  $\lim_{t\to \infty}E_t(\psi_{B_k^i})=\infty$. If  there exists $K>0$ and $0<\alpha \le 1$ such that $\|\overline\psi_{n,B_k^i,0}\|_\alpha<K$ for all $n\in \N_0$, then
$$\lim_{t\to \infty}\frac1{E_t(\psi_{B_k^i})}\int_0^t \mathbbm{1}_{U_s\cap B_k^i}\circ \F_s (x, u)~ds =1 \quad\text{ for } \mu_\F\text{-a.e. } (x, u)\in B^{r_2}.$$
\label{prop:restrict T2 flow}
\end{theorem}

\begin{proof}
We use the idea of Step 1 of the proof of \cite[Theorem 2]{GNO}.  
We will show 
$$\lim_{t\to \infty}\frac1{E_t(\psi_{B_k^i})}\int_0^t \mathbbm{1}_{U_s\cap B_k^i}\circ \F_s (x, 0)~ds =1 \quad\text{ for } \mu_2\text{-a.e. } (x, 0)\in B,$$ as then the proof for  $\mu_\F\text{-a.e. } (x, u)\in B^{r_2}$ follows.

We already know from Proposition~\ref{prop:SBC_T2} that for $\mu_2$-a.e.
 $x\in B$, 
\[
\frac{\sum_{j=0}^{n-1}\overline\psi_{j, B_k^i,0}(\T_2^j x)}{E_n(\overline\psi_{B_k^i,0})}\to 1\text{ as } n\to \infty.
\]
where $E_n (\overline\psi_{B_k^i,0}): =\sum_{j=0}^{n-1} \mu_2 (\overline{\psi}_{j,B_k^i, 0})$. Moreover the fact that  $E_n (\overline\psi_{B_k^i,0})\to \infty$ is equivalent to the divergence assumption in the statement of our theorem, as the sets are shrinking.
Lemma~\ref{lem:decr seq} controls  the effect of this perturbation in the limit when we switch on the $\eps$ parameter in one of the occurrences of $\psi_{n, B_k^i,\eps}$ above which deals with the shrinking of the balls during the flow between returns to the base. Note that $\mu_2 (\overline{\psi}_{n,B_k^i, \eps}) \le \mu_2 (\overline{\psi}_{n,B_k^i, -\eps})$ and $\lim_{\eps\to 0}
\frac{\mu_2 (\overline{\psi}_{j,B_k^i, \eps})}{\mu_2 (\overline{\psi}_{j,B_k^i, -\eps})}\to 1$ uniformly in $n$. Thus  
\[
\lim_{\eps\to 0} \frac{E_{q(t,x)}(\overline\psi_{B_k^i,+\eps})}{E_{q(t,x)}(\overline\psi_{B_k^i,-\eps})} \to 1
\]
uniformly in $n$. We will use these observations to squeeze $\frac{\int_0^{r_2^{q(t, x)}(x)} \psi_{B_k^i, s}\circ \F_s (x, 0)~ds}{E_{q(t,x)}(\overline\psi_{B_k^i,0})}$ between two corresponding convergent scaled Birkhoff sums.

Given $x\in B$, define $q(t, x)$ as the integer for which 
$$r_2^{q(n, x)}(x)\le t<r_2^{q(n, x)+1}(x)$$
where $r_2^m (x) = r_2(x) +r_2 (\T_2 x)+\ldots + r_2 (\T_2^{m-1} x)$.  
Observe that since the difference of the integral of $\mathbbm{1}_{U_s\cap B_k^i}\circ \F_s (x, \cdot)$ between times $r_2^{q(t, x)}(x)$ and $t$ is made up by at most one passage through $B_k^i$ which integrates to at most the length of $B_k^i$ in the vertical direction, i.e., 1, we have

\[
\int_0^t \psi_{B_k^i, s}\circ \F_s (x, 0)~ds-\int_0^{r_2^{q(t, x)}} \psi_{B_k^i, s}\circ \F_s (x, 0)~ds \le 1.
\]
Hence this difference is uniformly bounded independently of $x$ and $t$.

Thus $ \sum_{j=0}^{q(t,x)-1}\overline\psi_{j, B_k^i,\eps}(\T_2^j x) \le \int_0^{r_2^{q(t, x)}(x)} \psi_{B_k^i, s}\circ \F_s (x, 0) \le \sum_{j=0}^{q(t,x)-1}\overline\psi_{j, B_k^i,-\eps}(\T_2^j x)+1$.

So by Lemma~\ref{lem:av flow time}, for all small $\eps>0$, 
\begin{align*}
&\left(\frac{\sum_{j=0}^{q(t,x)-1}\overline\psi_{j, B_k^i,\eps}(\T_2^j x)}{E_{q(t,x)}(\overline\psi_{B_k^i,\eps})}\right)\left(\frac{E_{q(t,x)}(\overline\psi_{B_k^i,\eps})}{E_{q(t,x)}(\overline\psi_{B_k^i,0})}\right)\\
&\quad\le \frac{\int_0^{r_2^{q(t, x)}(x)} \psi_{B_k^i, s}\circ \F_s (x, 0)~ds}{E_{q(t,x)}(\overline\psi_{B_k^i,0})}\\
&\quad\le \left(\frac{\sum_{j=0}^{q(t,x)-1}\overline\psi_{j, B_k^i,-\eps}(\T_2^j x)+1}{E_{q(t,x)}(\overline\psi_{B_k^i,-\eps})}\right)\left(\frac{E_{q(t,x)}(\overline\psi_{B_k^i,-\eps})}{E_{q(t,x)}(\overline\psi_{B_k^i,0})}\right)
\end{align*}
  Then Lemmas~\ref{lem:av flow time},~\ref{lem:decr seq} and the fact that
\[
\lim_{\eps\to 0} \frac{E_{q(t,x)}(\overline\psi_{B_k^i,+\eps})}{E_{q(t,x)}(\overline\psi_{B_k^i,-\eps})}  \to 1
\]
 imply that 
\[
\lim_{t\to \infty} \frac{\int_0^{r_2^{q(t, x)}(x)} \psi_{B_k^i, s}\circ \F_s (x, 0)~ds}{E_{q(t, x)}(\overline\psi_{B_k^i,0})} =
 \lim_{t\to \infty} \frac{\int_0^{t} \psi_{B_k^i, s}\circ \F_s (x, 0)~ds}{E_{q(t, x)}(\overline\psi_{B_k^i,0})} =1.
\]

To complete the proof of the proposition, as in Step 2 of the proof of \cite[Theorem 2]{GNO}, we show that 
$$\lim_{n\to\infty}\frac{E_{n}(\psi_{B_k^i})}{E_{\lfloor n/\overline r_2\rfloor}(\overline\psi_{B_k^i,0})}=1.$$
Notice that this is the one part where our proof is easier than theirs since the flow is a first return to the base (this also accounts for the fact that Step 3 of that proof is unnecessary here).

By Lemma~\ref{lem:av flow time}, $q(n, x)\sim \lfloor\frac{n}{\overline{r_2}}\rfloor$.  Hence 
\begin{align*}
E_{\lfloor n/\overline r_2\rfloor}(\overline\psi_{B_k^i,0})&=\sum_{j=0}^{\lfloor\frac{n}{\overline{r_2}}\rfloor-1} \int_B\overline\psi_{j,B_k^i, 0}(y)~d\mu_2(y) \\
&= \sum_{j=0}^{\lfloor\frac{n}{\overline{r_2}}\rfloor-1} \int_B\int_0^{r_2(y)} \psi_{B_k^i, j\overline{r_2}+s}\circ \F_s (y, 0)~ds~d\mu_2(y)\\
&\sim \sum_{j=0}^{\lfloor\frac{n}{\overline{r_2}}\rfloor-1} \mu_\F (U_{j(\overline{r_2})}\cap B_k^i).
\end{align*}
Applying Lemma 6.2 with a speeded up time variable, we obtain  $ \sum_{j=0}^{\lfloor\frac{n}{\overline{r_2}}\rfloor-1} \mu_\F (U_{j(\overline{r_2})}\cap B_k^i) \sim
\int_0^{\frac{n}{\overline{r_2}}} \mu_\F (U_{s\overline{r_2}} \cap B_k^i)\overline{r_2} ~ds$, so a change of variables then gives 
$E_{\lfloor n/\overline r_2\rfloor}(\overline\psi_{B_k^i,0})\sim E_n (\psi_{B_k^i})$, thus completing the proof.\end{proof}

\subsection{An application of Theorem ~\ref{prop:restrict T2 flow}} One of the challenges in proving Borel-Cantelli lemmas when moving from the discrete system to the flow is that the induced characteristic functions are not, in general, characteristic functions.  In this subsection we prove that characteristic functions of balls  in the flow space induce observables which are sufficiently regular that we can apply Theorem~\ref{prop:restrict T2 flow} to them. In fact the averaging in the flow
direction regularizes functions.
If $(z,u)\in  B^{r_2}$ we let $B_{\eta}(z,u)$ denote a ball of radius $\eta$ about $(z,u)$ in the Euclidean metric $d_1((z,u),(z',u'))=[(u-u')^2+\sum_{j=1}^d(z_j-z_j^{'})^2]^{\frac{1}{2}}$. It is clear from our proof below other Euclidean metrics may be used, for example 
$d_2((z,u),(z',u'))=|u-u'|+\sum_{j=1}^d |z_j-z_j^{'}|$.



\begin{theorem}
Let $\delta(s)$ be a  decreasing sequence. For $\mu_\F$-a.e.\ $(z,u)\in B^{r_2}$ setting $U_s=B_{\delta(s)}(z,u)$, if $\lim_{t\to \infty} E_t(U)=\infty$ then
$$\lim_{t\to \infty}\frac1{E_t(U)}\int_0^t \mathbbm{1}_{U_s}\circ \F_s(x, v)~ds =1, \quad\text{ for } \mu_\F\text{-a.e. } (x, v)\in B^{r_2}$$
\label{prop:BC flow eg}
\end{theorem}

\begin{proof}  
As before we define 
\[
\psi_{B_k^i, s} := \mathbbm{1}_{U_{\delta(s)}\cap B_k^i},
\]
where 
\[
B_k^i:=\Big\{(x, t)\in B_k\times \R_+:i\le t<\min\{i+1, r_2(x)\}\Big\}.
\]
For large $s$ the ball  $B_{\delta(s)}(z,u)$ lies inside a fixed $B_{k^*}^{i^*}$ for some specific $k^*$, $i^*$. Since we have freedom to induce on a set $B$
placed anywhere in $\Delta$ we need not worry about $(z,u)$ lying on the boundary of a $B_k^{i}$. 

For $\gamma >0$ we also define  the induced function 
\[
\psi_n:=\psi_{n, B_{k^*}^{i^*}, \gamma}(x):=\int_0^{r_2(x)}\left(\mathbbm{1}_{U_{n(\overline{r_2}+\gamma)+s}}\cdot \mathbbm{1}_{B_{k^*}^{^*}i}\right)\circ \F (x, s)~ds,
\]

We have to show that there exists an $\alpha$ and a constant $K$ such that $\| \psi_n\|_{\alpha}<K$ for all $n$.

It suffices to show that there exist $\alpha$, $K$ such that
\[
\eps^{-\alpha}\int_Bosc(\psi_n, B_\eps (x))~dx< K
\]
for all $n$.

If $\delta(n(\overline{r_2})) \le \eps$ then  $osc(\psi_n, B_\eps(x)) \le 2 \eps$. This is because for each $y \in B_\eps(x)$, 
\[
\int_0^{r_2(x)}\left(\mathbbm{1}_{U_{(n(\overline{r_2}+\gamma)+s)}}\cdot \mathbbm{1}_{B_k^i}\right)\circ \F (x, s)~ds \le \delta( |n(\overline{r_2})|)
\le \eps.
\]

So we need only consider the supremum over small  $\eps < \delta(n(\overline{r_2}))$.
The ball $B_{\delta(s)}(z,u) \subset B_{k^*}^{i^*} $ lies in a $d+1$-dimensional Euclidean space. 
Its projection onto the $d$-dimensional space  $B$ is a ball  $B_{\delta(s)} (z)$ in $B_{k^*}$.
If the distance of $B_{\eps} (x)$ to $ B_{\delta(s)}(z))$ is greater than $2\eps$ then 
either $B_{\eps}(x) $ is in the exterior of $ B_{\delta(s)}(z)$  or $B_{2\eps} (x) \subset  B_{\delta(s)}(z)$. In the first case 
$\int_Bosc(\psi_n, B_\eps (x))=0$ as the flow starting in $B_{\eps}(x) $ does not meet  $B_{\delta(s)}(z,u)$. In the second case 
i.e.  $B_{\eps} (x)$ is bounded  away from the boundary of $B_{\delta(s)}(z))$ by $\eps$,
 then the two parts of the  boundary of $B_{\delta(s)}(z,u)$ which project to 
$B_{\eps} (x)$ may be written locally  as graphs over $B_{\eps} (x)$, the `height'  functions are given by 
$s-u=\sqrt{\delta (s)-\sum_{j=1}^d (t_j-z_j)^2}$  and $s-u=-\sqrt{\delta (s)-\sum_{j=1}^d (t_j-z_j)^2}$ respectively,
where $t=(t_1,\ldots,t_d)$ and $z=(z_1,\ldots,z_d)$ are Euclidean co-ordinates in $B$. Here we are restricting to $t$ satisfying
$\sqrt{ \sum_{j=1}^n (t_j-x_j)^2} <\eps$ where $x=(x_1,\ldots, x_d)$ is the center of $B_{\eps} (x)$.
Note that for both branches $|\frac{\partial s}{\partial t_i}|=\frac{1}{2} (\delta(s)-\sum_{j=1}^d (t_j-z_j)^2)^{-\frac{1}{2}}(2|t_i -z_i |)$.
In particular since $t$ satisfying  $\sqrt{ \sum_{j=1}^n (t_j-x_j)^2} <\eps$ is bounded
from the boundary of $B_{\delta(s)} (z)$ by $\eps$, i.e. $\sqrt{(\delta(s)-\sum_{j=1}^d (t_j-z_j)^2)}>\eps$
we have $|\frac{\partial s}{\partial t_i}| \le \frac{C}{\sqrt{\eps}}$ for all $i$ and hence
the oscillation  of  $\psi_n$ over $ B_\eps (x)$ is $O(\sqrt{\eps})$. Finally if  $ B_{\eps} (x)$ is within
$2\eps$ of  the boundary of $B_{\delta(s)} (z)$ then the oscillation  of  $\psi_n$ over $ B_{\eps} (x)$ is $O(1)$ but the 
$\mu_2$ measure of points $x$ within   a $2\eps$ neighborhood of the   boundary of $B_{\delta(s)} (z)$ is $O(\eps)$.  

Thus taking $\alpha=\frac{1}{2}$  there exists $K$ such that
\[
\eps^{-\frac{1}{2}}\int_Bosc(\psi_n, B_\eps (x))~dx< K
\]
for all $n$.
\end{proof}

\subsection{Borel-Cantelli lemmas for the Teichm\"{u}ller flow}

In this section, we prove Borel-Cantelli lemmas for the Teich\"{u}ller flow $\T$ seen as a suspension flow over the map $\mathcal{R}_2 : B_{\H_{(1)}} \to B_{\H_{(1)}}$ with roof function $r_2$.

We first prove a similar result for the map $\mathcal{R}_2$. Recall that this map preserves the measure $m_{B_{\H_{(1)}}}$ and is a skew-product over the map $\T_2 : B \to B$, which preserves $\mu_2$. To simplify the notations, we set $\mu := \mu_2$ and $\hat{\mu} := m_{B_{\H_{(1)}}}$.

\begin{proposition} \label{prop:bc_invert}

Let $(U_n)$ be a decreasing sequence of nested balls centered at a point $(x, \tau) \in B_{\H_{(1)}}$, with $\sum_n \hat{\mu}(U_n) = \infty$. Assume there exist $C >0$ and $\gamma > 0$ such that $\hat{\mu}(U_n) \ge C n^{- \gamma}$ and $(\log n) \mu(U_n) \le C$ for all $n \ge 0$. Then the sequence $(U_n)$ is strong Borel-Cantelli for $\mathcal{R}_2$.

\end{proposition}

\begin{proof}
We follow the proof of \cite[Theorem 1.5]{Zhang}. Let $f_k = \mathbbm{1}_{U_k} \circ \mathcal{R}_2^k$. We denote by $E(.)$ the expectation operator with respect to $\hat{\mu}$. We trivialize $B_{\H_{(1)}}$ to a product via the natural diffeomorphism $B_{\H_{(1)}} \to B \times \mathbb{P} T_B^+$, where $\mathbb{P} T_B^+$ is the image of $T_B^+$ in the projective space $\mathbb{P} \R^{\A}$. Let $\Pi_x$ and $\Pi_\tau$ be the projections on the factors $B$ and $\mathbb{P} T_B^+$ respectively. We denote by $m_1$ the Lebesgue measure on each factor, and by $m_2$ the product Lebesgue measure on $ B \times \mathbb{P} T_B^+$. The measure $\hat{\mu}$ has a smooth density with respect to $m_2$, which is bounded uniformly from above and below. Let $E(.)$ be the expectation operator with respect to the measure $\hat{\mu}$.

For $i < j$, we calculate 
\[
\begin{aligned}
E(f_i f_j) & = \int \mathbbm{1}_{U_i} \circ \mathcal{R}_2^i \, \mathbbm{1}_{U_j} \circ \mathcal{R}_2^j \, d \hat{\mu} = \int \mathbbm{1}_{U_i} \, \mathbbm{1}_{U_j} \circ \mathcal{R}_2^{j-i} \, d \hat{\mu} \\
& \lesssim \int_{U_i} \mathbbm{1}_{\Pi_x U_i} \, \mathbbm{1}_{\Pi_x U_j} \circ \Pi_x \circ \mathcal{R}_2^{j-i} \, d m_2 \\
& \lesssim m_1(\Pi_\tau U_i) m_1( \Pi_x U_i \cap \T_2^{-(j-i)} \Pi_x U_j) \\
& \lesssim m_1(\Pi_\tau U_i) \mu(\Pi_x U_i \cap \T_2^{-(j-i)} \Pi_x U_j) \\
& \lesssim m_1(\Pi_\tau U_i) \left( \mu(\Pi_x U_i) \mu(\Pi_x U_j) + C \theta^{j-i} \mu(\Pi_x U_j) \right) \\
& \lesssim  m_1(\Pi_\tau U_i) \left( m_1(\Pi_x U_i) m_1(\Pi_x U_j) + C \theta^{j-i} m_1(\Pi_x U_j) \right) \\
& \lesssim (m_2(U_i))^{\frac 1 2} \left((m_2(U_i))^{\frac 1 2} (m_2(U_j))^{\frac 1 2} + C \theta^{j-i} (m_2(U_j))^{\frac 1 2} \right) \\
& \lesssim (m_2(U_i))^{\frac 3 2} + \theta^{j-i} m_2(U_i).
\end{aligned}
\]

Throughout this calculation, we have used the fact that $\mu$ and $\hat{\mu}$ have a density with respect to $m_1$ and $m_2$ respectively which are bounded uniformly from above and below, decay of correlations for $\T_2$ given by Proposition \ref{prop:decay_gibbs} and the fact that there exists a constant $K$ such that for all ball $U$, $m_1(U) \le K (m_2(U))^{\frac 1 2}$.

So, using decay of correlations for $\mathcal{R}_2$ and Lipschitz observables, we have 
\begin{eqnarray*}
\lefteqn{\sum_{j=i+1}^n (E(f_i f_j)-E(f_i)E(f_j)) \le ( \sum_{j=i+1}^{i+a\log{i}} + \sum_{j>i + a\log{i}}) [E(f_i f_j)-E(f_i)E(f_j)] }\\
&\lesssim  (\log i) (m_2(U_i))^{\frac{3}{2}}  +  m_2(U_i) +  \sum_{j>i + a\log{i}}   \alpha^{j-i} ||\tilde{f}_i||_{\rm Lip} ||\tilde{f}_j||_{\rm Lip} \\
\end{eqnarray*}
where $a$ will be chosen later and  $\tilde{f}_i$ is a  Lipschitz approximation  to ${f}_i$, satisfying $m_2(|\tilde{f}_i-f_i|) \lesssim \frac{1}{i^2}$ 
and $\|\tilde{f}_i\|_{\rm Lip} \lesssim i^{\kappa}$ for some fixed $\kappa$. We are able to satisfy both conditions as
$m_2(U_i) \gtrsim i^{-\gamma}$ for some $\gamma>0$. We have $(\log i) (m_2(U_i))^{\frac{3}{2}}\lesssim m_2 (U_i)$ and for $a>0$ sufficiently large
$$\sum_{j>i + a\log{i}} \alpha^{j-i} ||\tilde{f}_i||_{\rm Lip} ||\tilde{f}_j||_{\rm Lip} \lesssim m_2 (U_i).$$

We have thus shown that \[
\sum_{i=m}^n \sum_{j=i+1}^n (E(f_i f_j)-E(f_i)E(f_j)) \lesssim  \sum_{i=m}^n E(f_i)
\]
which implies the strong Borel-Cantelli property by Proposition \ref{prop:sprindzuk}.
\end{proof}

\begin{remark}
Note that the proof above does not use the assumption that the balls are nested, nor that they are balls just that they may be approximated by Lipschitz functions $\tilde{f}_i$ such that
$m_2(|\tilde{f}_i-f_i|) \lesssim \frac{1}{i^2}$ 
and $\|\tilde{f}_i\|_{Lip} \lesssim i^{\kappa}$ for some fixed $\kappa$. 

\end{remark} 

We now show that the (SBC) property for  the map $\mathcal{R}_2$ implies the SBC property  for nested balls $U_t$   in the full suspension flow. 

\begin{theorem}
Let $U = (U_t)_{t \ge 0}$ be a family of shrinking balls in $B^{r_2}_{\H_{(1)}}$, with $\mu_\T(U_t) \gtrsim t^{- \gamma}$ for some $\gamma > 0$ and $\sup_{t \ge 0} (\log t) \mu_\T(U_t) < \infty$.
Assume that 
\[
E_t:= E_t(U) = \int_{0}^t \mu_{\T} (U_s)ds
\]
diverges. 

Then the family $U$ is strong Borel-Cantelli for the flow: for $\mu_\T$ a.e. $p \in B_{\H_{(1)}}^{r_2}$, $$\frac{1}{E_t(U)} \int_{0}^{t} \mathbbm{1}_{U_s}(\T^t(p)) \, ds \to 1.$$
\end{theorem}

\begin{proof}
Note that the measure on the flow $\mu_{\T}$  is the product of the base measure and 
Lebesgue measure in  the flow direction, so that $d\mu_\T =d\hat{\mu} \times dt$ and that the projection $\Pi$, say,  via flow lines of the balls $U_t$ in the suspension flow is a $t$-parametrized sequence of nested `balls'  $C_t$ in the Poincar\'e section $B_{\H_{(1)}}$. The dynamics 
of  the return map to $B_{\H_{(1)}}$ is given by  the skew-product  map $\mathcal{R}_2: B_{\H_{(1)}} \to B_{\H_{(1)}}$.  The flow  $(\T^t)$ is rectifiable in a sufficiently small neighborhood of the balls $U_t$. 
Let $\hat{k} (p)$ be the time that $\T^t(p)$ returns to $B_{\H_{(1)}}$ for the $k$-th time under $\T$, where $p\in B_{\H_{(1)}}$  or $\hat{\mu} $ a.e. $p \in B_{\H_{(1)}}$,
\[
\lim_{k\to \infty} \frac{\hat{k} (p) }{k}   =\int_{B_{\H_{(1)}}}  r_2 \,d \hat{\mu}:=\bar{r}_2
\]

 We fix an integer $n$ and  discretize $C_t$ into disjoint sets $C_{t,j}$, $j=1$ to $n$, of roughly equal  $\hat{\mu}$ measure and define  $\tilde{U}_{t,j}:=\{ q \in U_t: \Pi q \in C_{t,j} \}$. Hence 
 $C_{t,j}$ lie in $B_{\H_{(1)}}$ while $\tilde{U}_{t,j}$ lies in the full suspension flow $B_{\H_{(1)}}^{r_2}$.

We consider two sequences of sets $C_{\alpha,t,j}$  and 
$C_{\beta, t , j}$ in the suspension flow defined by  flow lines through  $C_{t,j}$  of constant length $\tau_1 (t,j)$ and $\tau_2 (t,j)$ such that for each $\tilde{U}_{t,j}$,
$ C_{\alpha,t,j} \subset \tilde{U}_{t,j} \subset  C_{\beta,t,j} $ and moreover for each $j,t>0$,  $\mu_\T  (C_{\beta,t,j} ) - \mu_\T  (C_{\alpha,t,j} ) \le e(n) \mu_\T (\tilde{U}_{t,j})$ where $e(n)\to 0$ 
as $n\to \infty$.  We can ensure this as the boundary of $\tilde{U}_{t,j}$ consists of two manifolds, each a smooth graph over $C_{t,j}$. The role of the sequence of sets $C_{\alpha,t,j}$, $C_{\beta,j,t}$
is to provide  discretized lower  and upper bounds between which we can squeeze the continuous  flow.

Hence $\mu_\T   (\cup_{j} C_{\alpha,t,j} ) \le \mu_\T (U_t) \le \mu_\T  (\cup_{j} C_{\beta,t,j} )$ and 
 $\mu_\T  (\cup_{j} C_{\beta,t,j} ) -\mu_\T  (\cup_{j} C_{\alpha,t,j} ) \le e(n) \mu_\T (U_t)$ where $e(n)\to 0$ 
as $n\to \infty$. 

Recall  $\hat{k}(p)$ denotes the k-th return time to $B_{\H_{(1)}}$ of a point $p\in B_{\H_{(1)}}$ under the flow $\T^t$ so that $\T^{\hat{k}} (p)=\mathcal{R}_2^k (p)$. By the ergodic theorem given
$\eps >0$ for $\hat{\mu}$ a.e. $p$ there  exists $k^* (\eps)(p) $ such that $k(\bar{r}_2 -\eps)\le \hat{k} (p) \le k (\bar{r}_2 +\eps)$ for all $k>k^*(\eps)$.

 We fix $\eps$ and  $n$. We let $[\alpha]$ denote the integer part of the real number $\alpha$. For each $j$, the sequences of sets, indexed by $k$,   $(C_{[k(\bar{r}_2 +\eps)],j})$ and  $(C_{[k(\bar{r}_2 -\eps)],j})$ both have the (SBC) property for $\mathcal{R}_2 : B_{\H_{(1)}} \to B_{\H_{(1)}}$, i.e.
 \[
 \lim_{k\to \infty} \frac{1}{E_{(k,j,\eps,+)}} \sum_{i=1}^k \mathbbm{1}_{ C_{([i(\bar{r}_2 +\eps)],j)}} \circ \mathcal{R}_2^i (p)=1
 \]
for $\hat{\mu}$ a.e. $p \in B_{\H_{(1)}}$, where $E_{(k,j,\eps,+)}:=\sum_{i=1}^k \hat{\mu}(C_{[i(\bar{r}_2+\eps)],j})$ and similarly for $(C_{([k(\bar{r}_2 -\eps)],j)})$. Indeed, this follows from Proposition \ref{prop:bc_invert} since $\hat{\mu}(C_{[k(\bar{r}_2 + \eps)],j}) \sim \mu_\T(U_{[k(\bar{r}_2 + \eps)]})^{1-\frac{1}{d}}$ as $k \to \infty$, for fixed $n$ and $\eps$.
 
 Note that $k(\bar{r}_2 -\eps)\le \hat{k} \le k (\bar{r}_2 +\eps)$ and by the Lipschitz regularity  of $\hat{\mu} (C_{t,j})$ in $t$ if $k(\bar{r}_2 -\eps)\le t \le k (\bar{r}_2 +\eps)$ 
 then $\hat{\mu}( C_{([k(\bar{r}_2 +\eps)],j)} )-  \hat{\mu}(C_{([k(\bar{r}_2 -\eps)],j)} ) \le \rho (\eps) \hat{\mu} (C_{([k(\bar{r}_2 +\eps)],j)})$ where $\rho (\eps)\to 0$ as $\eps\to 0$. 
 
 Furthermore, for sufficiently large $t$,  once $\mathcal{R}_2^k (p)$ enters $C_{t,j}$ its trajectory spends a length of flow time between  $\tau_1( [k(\bar{r}_2-\eps)],j)$ and  $\tau_2( [k(\bar{r}_2+\eps)],j)$ in the sets $(\tilde{U}_{t,j})$.

 Thus for $\hat{\mu}$ a.e. $p$, (recall $n$ is fixed)
 \[
\begin{aligned}
 \sum_{j=1}^{n}\sum_{i=1}^T \tau_1( [i(\bar{r}_2-\eps)],j)) \hat{\mu} ( C_{([i(\bar{r}_2 -\eps)],j)} ) & \le\sum_{j=1}^n \int_{0}^{T \bar{r}_2}  \hat{\mu}(\Pi \tilde{U}_{t,j}) 1_{U_{t,j}} \circ \T^t(p) dt  \\ &\le  \sum_{j=1}^n\sum_{i=1}^T \tau_1( [i(\tau_1+\eps),j)) \hat{\mu} ( C_{([i(\tau_1 +\eps],j)} )
\end{aligned}
 \]
  
  The sums  $L(T,n):= \sum_{j=1}^{n}\sum_{i=1}^T \tau_1( [i(\bar{r}_2-\eps)],j)) \hat{\mu} ( C_{([i(\bar{r}_2 -\eps)],j)} ) $ and   $U(T,n):=\sum_{j=1}^n\sum_{i=1}^T \tau_1( [i(\tau_1+\eps),j)) \hat{\mu} ( C_{([i(\tau_1 +\eps],j)} )$ are Riemann sums, and $\lim_{T \to \infty} \frac{U(T,n)}{L(T,n)} =\kappa (n)$ where $\kappa (n)\to 1$ as $n\to \infty$.

  Using a change of variables $$\sum_{j=1}^n \int_{0}^{T \bar{r}_2}  \hat{\mu}(\Pi \tilde{U}_{t,j}) 1_{U_{t,j}} \circ \T^t (p) dt \sim \frac{1}{\bar{r}_2} \sum_{j=1}^n \int_{0}^{T}  \hat{\mu}(\Pi \tilde{U}_{t,j}) 1_{U_{t,j}} \circ \T^t (p) dt$$ where $H(T)\sim G(T)$ means $\lim_{T\to \infty} \frac{G(T)}{H(T)}=1$.

  Furthermore 
  $$\left|\frac{\frac{1}{\tau_1} \sum_{j=1}^n \int_{0}^{T} \hat{\mu}(\Pi \tilde{U}_{t,j}) 1_{U_{t,j}} \circ \T^t(p) dt}{\int_0^T \nu (U_t) dt}-1 \right|\le \kappa_2(n)$$ where $\kappa_2 (n)\to 0$ as $n\to\infty$.
  


This proves the SBC property for nested balls in the full suspension flow.
\end{proof}

\subsection{Extreme Value Laws for the flow}
We have established EVLs for sufficient regular observations on the dynamical system $(\T_2,B,\mu_2)$. We now consider 
EVLs for  the flow $\F_s: B^{r_2} \to B^{r_2}$. To do this we use~\cite[Theorem 2.6]{HNT} which relates Extreme Value Theory for
functions on the suspension of a base transformation to the Extreme Value statistics of   observations on the base.

We start with some preliminary notation.  Let $\overline{r_2}=\int_B r_2 (x) d\mu_2$. Let $\phi: B^{r_2} \to \R \cup \{+\infty\}$ be a function, 
strictly maximized at a point $(x_0,u_0)\in B^{r_2}$, 
which is sufficiently regular that for large $r$ the set $\{ (x,u) \in B^{r_2}: \phi ((x,u) )>r\}$ corresponds to a topological ball centered at $(x_0,u_0)$. 
Let $\bar{\phi}(x)=\sup_{0\le u\le r_2 (x)} \phi ((x,u) )$ and define $u_n (t)$ by the requirement that $n\mu_2 \{\bar{\phi} >  u_n (t) \}\to t$.
Let $M_T (x,s):= \max \{ \phi (F_s (x,u): 0\le s\le T\}$.  As a consequence of \cite[Theorem 2.6]{HNT},

\begin{proposition}
Suppose when we write $u_n(t)=\frac{g(t)}{a_n}+b_n$  the normalizing constants $a_n>0$ and $b_n$ satisfy:
  \begin{align}
    \lim_{\eps\to 0} \limsup_{n\to\infty} \;
    a_{n}|b_{[n+\eps n]}-b_n| &= 0, \label{const_1} \\
    \lim_{\eps\to 0} \limsup_{n\to\infty}\left|1-\frac{a_{[n+\eps
          n]}}{a_n}\right| &=0.\label{const_2}
  \end{align}
 
Then,

1) If $x_0$ is not a periodic point for $\T_2$ then $\mu \{ M_T \le u_{[T/\overline{r_2}]} (t) \}\to e^{-t}$.

(2) If $x_0$ is a  repelling periodic point of prime period $k$ then $\mu \{ M_T \le u_{[T/\overline{r_2}]} (t) \}\to e^{-\theta t}$ where $\theta=1-|Jac(D\T_2^{-k}) (p_0)|$.

\end{proposition}

The extreme value result for the Teichm\"{u}ller flow $\T = (\T^t)_{t \in \R}$ holds from combining ~\cite[Theorem 2.1]{Gupta} with~\cite[Corollary 2.3]{HNT} (note that the proof for Gibbs Markov maps holds in any dimension as long as con formality holds)
and~\cite[Theorem 2.6]{HNT}.

\section{Appendix: Aperiodicity and weak mixing}

Let $(X, T, \mu)$ be an ergodic measure-preserving dynamical system. 

\begin{definition} 
$(X, T, \mu)$ is \emph{weakly mixing} if $f \circ T = e^{it} f$ for some non-zero $f \in L^2(\mu)$ and $t \in [0, 2 \pi )$ implies that $t = 0$ and $f$ is constant.
\end{definition}

\begin{remark} This definition is equivalent to the classical one, stating that $$\frac{1}{n} \sum_{k=0}^{n-1} \left| \mu(T^{-k}(A) \cap B) - \mu(A) \mu(B) \right| \to 0$$ for any measurable sets $A$ and $B$. See \cite[Theorem 1.26]{Wal} in the case where $(X, T, \mu)$ is invertible, and \cite[Theorem 664]{KalMcCut} or \cite[Theorem 2.36]{EinWa} for a proof of the equivalence valid in any case.

 \end{remark}

Let $Y \subset X$ be a subset of positive $\mu$-measure. We denote by $\tau(y)$ the first return time of $y \in Y$ to $Y$: $$\tau(y) = \min \{ n \ge 1 \, :\, T^n y \in Y \}.$$
We then define the first return map $\hat{T} : Y \to Y$ by $\hat{T} = T^{\tau}$. It preserves the normalisation $\mu_Y$ of the restriction to $Y$ of the measure $\mu$ and is ergodic with respect to it.

\begin{definition} 
We will say that the first return time is \emph{aperiodic} if $f \circ \hat{T} = e^{i t \tau} f$ for some non-zero $f \in L^2(\mu_Y)$ and $t \in [0,2\pi)$ implies that $t = 0$ and $f$ is constant.

\end{definition}

\begin{remark} By \cite[Proposition 1.1]{Mor}, the relation $f \circ \hat{T} = e^{i t \tau} f$ is equivalent to $\mathcal{L}(e^{i t \tau} f) = f$, where $\mathcal{L}$ is the transfer operator of $\hat{T}$ with respect to the measure $\mu_Y$.
\end{remark}

\begin{proposition} \label{prop:aperiodicity}
The first return time is aperiodic if and only if $(X,T,\mu)$ is weakly mixing.
\end{proposition}

\begin{proof}
Suppose first that the first return time is aperiodic and let $f \in L^2(\mu)$ non-zero and $t \in [0,2\pi)$ such that $f \circ T = e^{it} f$. We easily verify that the restriction $f_Y$ of $f$ to $Y$ satisfies $f_Y \circ \hat{T} = e^{it \tau} f_Y$: $$f_Y(\hat{T} y) = f(T^{\tau(y)} y) = e^{i t \tau(y)} f(y) = e^{it \tau(y)} f_Y(y).$$ $f_Y$ is also non identically zero: otherwise, $f$ would vanish on the set $\cup_{n \ge 0} T^{-n} Y$, which by ergodicity is equal to $X$ mod $\mu$.  Aperiodicity yields that $t = 0$, which means that $f \circ T = f$. Ergodicity implies that $f$ is constant. \\ Conversely, suppose that $(X, T , \mu)$ is weakly mixing and that $f \in L^2(\mu_Y)$ is non identically zero and satisfies $f \circ \hat{T} = e^{it \tau} f$. We first extend $\tau$ on the whole space $X$ as being the first hitting time. By ergodicity, it is well defined $\mu$-a.e. We then define $\tilde{f} \in L^2(\mu)$ by $\tilde{f} = e^{-it \tau} f \circ T^{\tau}$. Since $T^{\tau(x)}x$ belongs to $Y$ for $\mu$-a.e.\ $x \in X$ by definition, $\tilde{f}$ is well-defined. Our assumption on $f$ implies that $\tilde{f}$ and $f$ coincide on $Y$, so that it is non identically zero. \\ Now, we verify that $\tilde{f} \circ T = e^{it} \tilde{f}$. Let $x \in X$ with $\tau(x) > 1$. Since $\tau$ is the {\em first} hitting time, we have $\tau(Tx) = \tau(x) - 1$. Hence, $\tilde{f}(Tx) = e^{-it \tau(Tx)} f(T^{\tau(Tx)}Tx) = e^{it}e^{-it \tau(x)}f(T^{\tau(x)}x) = e^{it}\tilde{f}(x)$. If $\tau(x) = 1$, which implies $Tx \in Y$, we have by definition of $\tilde{f}$ that $\tilde{f}(x) = e^{-it}f(Tx) = e^{-it}\tilde{f}(Tx)$. \\ Weak mixing implies that $t = 0$ and $\tilde{f}$ is constant. Since the restriction of $\tilde{f}$ to $Y$ is $f$, this shows that $f$ is constant, and concludes the proof.
\end{proof}

\end{document}